\documentclass[reqno, 11pt, a4paper]{amsart}

\usepackage{latexsym,ifthen,xspace}
\usepackage{amsmath,amssymb,amsthm}
\usepackage{enumerate, calc}
\usepackage[colorlinks=true, pdfstartview=FitV, linkcolor=blue,
  citecolor=blue, urlcolor=blue, pagebackref=false]{hyperref}
\usepackage[text={33pc,605pt},centering]{geometry}    
\usepackage{graphicx,color}

\newtheorem{theorem}{Theorem}[section]
\newtheorem{lemma}[theorem]{Lemma}
\newtheorem{proposition}[theorem]{Proposition}
\newtheorem{assumption}[theorem]{Assumption}
\newtheorem{corollary}[theorem]{Corollary}

\theoremstyle{definition}
\newtheorem{definition}[theorem]{Definition}
\newtheorem{example}[theorem]{Example}

\theoremstyle{remark}
\newtheorem{remark}[theorem]{Remark}

\numberwithin{equation}{section}

\DeclareMathAlphabet{\mathsl}{OT1}{cmss}{m}{sl}
\SetMathAlphabet{\mathsl}{bold}{OT1}{cmss}{bx}{sl}

\newcommand{\al}{\ensuremath{\alpha}}

\newcommand{\ga}{\ensuremath{\gamma}}
\newcommand{\de}{\ensuremath{\delta}}

\renewcommand{\th}{\ensuremath{\theta}}

\newcommand{\la}{\ensuremath{\lambda}}

\newcommand{\si}{\ensuremath{\sigma}}

\newcommand{\om}{\ensuremath{\omega}}

\newcommand{\vr}{\ensuremath{\varrho}}

\newcommand{\Si}{\ensuremath{\Sigma}}

\newcommand{\Om}{\ensuremath{\Omega}}
\newcommand{\cA}{\ensuremath{\mathcal A}}
\newcommand{\cB}{\ensuremath{\mathcal B}}

\newcommand{\cE}{\ensuremath{\mathcal E}}
\newcommand{\cF}{\ensuremath{\mathcal F}}

\newcommand{\cL}{\ensuremath{\mathcal L}}

\newcommand{\cN}{\ensuremath{\mathcal N}}

\newcommand{\cP}{\ensuremath{\mathcal P}}

\newcommand{\cS}{\ensuremath{\mathcal S}}

\newcommand{\bbE}{\ensuremath{\mathbb E}}

\newcommand{\bbN}{\ensuremath{\mathbb N}} 
 
\newcommand{\bbP}{\ensuremath{\mathbb P}} 
 
\newcommand{\bbR}{\ensuremath{\mathbb R}}

\newcommand{\bbZ}{\ensuremath{\mathbb Z}} 
\definecolor {orange} {rgb} {0.569, 0.259, 0.0}

\newcommand{\me}{\ensuremath{\mathrm{e}}}

\let\norm\undefined
\newcommand{\norm}[3]{%
  \ensuremath{%
    \big\lVert
      #1
    \big\rVert_{\raisebox{-.0ex}{$\scriptstyle \ell^{\raisebox{.2ex}{$\scriptscriptstyle #2$}} (#3)$}}
  }
}
\newcommand{\Norm}[2]{%
  \ensuremath{%
    \big\lVert
      #1
    \big\rVert_{\raisebox{-.0ex}{$\scriptstyle #2$}}
  }
}

\DeclareMathOperator{\mean}{\mathbb{E}}

\DeclareMathOperator{\prob}{\mathbb{P}}

\DeclareMathOperator{\var}{\mathbb{V}ar}
\DeclareMathOperator{\cov}{\mathbb{C}ov}

\DeclareMathOperator{\supp}{\mathrm{supp}}

\DeclareMathOperator{\osr}{\mathrm{osr}}
\DeclareMathOperator{\osc}{\mathrm{osc}}

\DeclareMathOperator{\lip}{\mathrm{lip}}

\def\indicator{{\mathchoice {1\mskip-4mu\mathrm l}%
{1\mskip-4mu\mathrm l}{1\mskip-4.5mu\mathrm l}%
{1\mskip-5mu\mathrm l}}}

\newcommand{\eps}{\epsilon}
\newcommand{\R}{\mathbb{R}}
\newcommand{\Z}{\mathbb{Z}}

\newcommand{\E}{\mathbb{E}}
\newcommand{\N}{\mathbb{N}}
\newcommand{\bP}{\mathbb{P}}
\newcommand{\bR}{\mathbb{R}}

\newcommand{\bZ}{\mathbb{Z}}
\newcommand{\threepartdef}[6]
{
	\left\{
		\begin{array}{lll}
			#1 & \mbox{if } #2 \\
			#3 & \mbox{if } #4 \\
			#5 & \mbox{if } #6
		\end{array}
	\right.
}
\usepackage{mathtools}
\DeclarePairedDelimiter\ceil{\lceil}{\rceil}
\DeclarePairedDelimiter\floor{\lfloor}{\rfloor}
\usepackage{commath}
\usepackage{color}
\usepackage{cite}
\usepackage{graphicx,bbm,enumerate,amsthm}

\let\norm\undefined
\newcommand{\norm}[1]{%
  \ensuremath{%
    \big\lVert
      #1
    \big\rVert}
}
\let\onorm\undefined
\newcommand{\onorm}[1]{%
  \ensuremath{%
    \big\lvert
      #1
    \big\rvert}
}

\begin{document}

\title[Local Limit Theorems for the RCM]{Local Limit Theorems for the Random Conductance Model and Applications to the Ginzburg-Landau $\nabla\phi$ Interface Model}


\author{Sebastian Andres}
\address{The University of Manchester}
\curraddr{Department of Mathematics,
Oxford Road, Manchester M13 9PL}
\email{sebastian.andres@manchester.ac.uk}
\thanks{}

\author{Peter A.\ Taylor}
\address{University of Cambridge}
\curraddr{Centre for Mathematical Sciences, Wilberforce Road, Cambridge CB3 0WB}
\email{pat47@cam.ac.uk}
\thanks{}

\subjclass[2000]{60K37; 60J35; 60F17; 82C41; 39A12}

\keywords{Random conductance model, local limit theorem, De Giorgi iteration, ergodic, stochastic interface model.}

\date{\today}

\dedicatory{}

\begin{abstract}
We study a continuous-time random walk on $\Z^d$ in an environment of random conductances taking values in $(0,\infty)$. For a static environment, we extend the quenched local limit theorem to the case of a general speed measure, given suitable ergodicity and moment conditions on the conductances and on the speed measure. Under stronger moment conditions, an annealed local limit theorem is also derived. Furthermore, an annealed local limit theorem is exhibited in the case of time-dependent conductances, under analogous moment and ergodicity assumptions. This dynamic local limit theorem is then applied to prove a scaling limit result for the space-time covariances  in the Ginzburg-Landau $\nabla\phi$ model. We also show that the associated Gibbs distribution scales to a Gaussian free field. These results apply to convex potentials for which the second derivative may be unbounded. 
\end{abstract}

\maketitle


\section{Introduction}
\label{section:intro}

\subsection{The Model}
We consider the graph $G=(\mathbb{Z}^d,E_d)$ of the hypercubic lattice with the set of nearest-neighbour edges $E_d:=\{\{x,y\}:x,y\in\mathbb{Z}^d,\lvert x-y\rvert =1\}$ in dimension $d\geq 2$. We place upon $G$ positive weights $\omega=\{\omega(e)\in(0,\infty) : e\in E_d\}$, and define two measures on $\mathbb{Z}^d$,
\begin{align*}
\mu^\omega(x):=\sum_{y\sim x}\omega(x,y),\qquad
\nu^\omega(x):=\sum_{y\sim x}\frac{1}{\omega(x,y)}.
\end{align*}

Let $(\Omega,\mathcal{F}):=(\mathbb{R}_+^{E_d},\mathcal{B}(\mathbb{R}_+)^{\otimes E_d})$ be the measurable space of all possible environments. We denote by $\mathbb{P}$ an arbitrary probability measure on $(\Omega,\mathcal{F})$ and $\mathbb{E}$ the respective expectation. The measure space $(\Omega, \cF)$ is naturally equipped with a group of space shifts $\big\{\tau_z : z \in \bbZ^d\big\}$, which act on $\Omega$ as
\begin{align} \label{eq:def:space_shift}
  (\tau_z \om)(x, y) \; := \; \om(x+z, y+z),
 \qquad \forall\, \{x,y\} \in E_d.
\end{align}

Let $\theta^\omega: \Z^d\rightarrow (0,\infty)$ be a positive function which may depend upon the \textit{environment} $\omega\in\Omega$. The random walk $(X_t)_{t\geq0}$ defined by the following generator,
$$\mathcal{L}^\omega_\theta f(x):=\frac{1}{\theta^\om (x)}\sum_{y\sim x}\omega(x,y)\left(f(y)-f(x)\right),$$
acting on bounded functions $f:\Z^d\rightarrow\R$, is reversible with respect to $\theta^\omega$, and we call this process the \textit{random conductance model (RCM)} with \textit{speed measure} $\theta^\omega$. We denote $P_x^\omega$ the law of this process started at $x\in\mathbb{Z}^d$ and $E_x^\omega$ the corresponding expectation. There are two natural laws on the path space that are considered in the literature - the quenched law $P_x^\omega(\cdot)$ which concerns $\mathbb{P}$-almost sure phenomena, and the annealed law $\mathbb{E} P_x^\omega(\cdot)$. 

If the random walk $X$ is currently at $x$, it will next move to $y$ with probability $\omega(x,y)/\mu^\omega(x)$, after waiting an exponential time with mean $\theta^\omega(x)/\mu^\omega(x)$ at the vertex $x$.
The main results of this paper are statements about the heat kernel of $X$, 
\begin{align*}
p_\theta^\omega(t,x,y) \; := \; \frac{P_x^\omega\left(X_t=y\right)}{\theta^\om(y)}, \qquad t\geq 0, \, x,y\in\Z^d.
\end{align*}
Perhaps the most natural choice for the speed measure is $\theta^\omega\equiv\mu^\omega$, for which we obtain the constant speed random walk (CSRW) that spends i.i.d. $\text{Exp}(1)$-distributed waiting times at all vertices it visits. Another well-studied process, the variable speed random walk (VSRW), is recovered by setting $\theta^\omega\equiv 1$, so called because as opposed to the CSRW, the waiting time at a vertex $x$ does indeed depend on the location; it is an $\text{Exp}(\mu^\omega(x))$-distributed random variable.

\subsection{Main Results on the Static RCM} \label{sec:intro_static}
As our first main results we obtain quenched and annealed local limit theorems for the static random conductance model.  A general assumption required  is stationarity and ergodicity of the environment.

\begin{assumption}
\label{ass:ergodicity general speed measure}
\begin{enumerate}[(i)]
\item $\mathbb{P}[0<\omega (e)<\infty]=1$ and $\mean[\om(e)]<\infty$ for all $e\in E_d$. 
\item $\mathbb{P}$ is ergodic with respect to spatial translations of $\mathbb{Z}^d$, i.e.\ $\mathbb{P}\circ \tau_x^{-1}=\mathbb{P}$ for all $x\in\mathbb{Z}^d$ and $\mathbb{P}(A)\in\{0,1\} \text{ for any } A\in\mathcal{F} \text{ such that }\tau_x(A)=A \text{ for all }x\in\mathbb{Z}^d.$
\item $\theta$ is stationary, i.e.\ $\theta^\om(x+y)=\theta^{\tau_y\omega}(x)$ for all $x,y\,\in\Z^d$ and $\bP$-a.e.\ $\omega\in\Omega$. Further, $\mathbb{E}[\theta^\om (0)]<\infty$ and $\mathbb{E}[\theta^\om (0)/\mu^\om(0)]\in(0,\infty)$.
\end{enumerate}
\end{assumption}
In particular, the last condition in Assumption~\ref{ass:ergodicity general speed measure}(iii) ensures that the process $X$ is non-explosive.
During the last decade, considerable effort has been invested in the derivation of \emph{quenched invariance principles} or \emph{quenched functional central limit theorems (QFCLT)}, see the surveys \cite{ Bi11, Ku14} and references therein. 
The following QFCLT for random walks under ergodic conductances is the main result of \cite{invariance}.

\begin{theorem}[QFCLT]
\label{thm:qfclt general speed}
Suppose Assumption \ref{ass:ergodicity general speed measure} holds. Further assume that there exist $p,\,q\in(1,\infty]$ satisfying $\frac{1}{p}+\frac{1}{q}<\frac{2}{d}$ such that
$\E\left[\omega(e)^p\right]<\infty$ and $\E\big[\omega(e)^{-q}\big]<\infty$
for any $e\in E_d$. For $n\in \bbN$, define $X_t^{(n)} := \frac{1}{n} X_{n^2 t}$, $t\geq 0$.  Then, for $\prob$-a.e.\ $\om$,   $X^{(n)}$ converges (under $P_0^\om$) in law towards a Brownian motion on $\bbR^d$  with a deterministic non-degenerate covariance matrix $\Si^2$.
\end{theorem}
\begin{proof}
For the VSRW, this is \cite[Theorem~1.3]{invariance}. As noted in \cite[Remark~1.5]{invariance} the QFCLT extends to the random walk with general speed measure $\theta^\omega$ provided $\E[\theta^\om(0)]\in(0,\infty)$. See \cite[Section 6.2]{andres2013invariance} for a proof of this extension in the case of the CSRW.
\end{proof}

Recently the moment condition in Theorem~\ref{thm:qfclt general speed} has been improved in \cite{BS20}.

\begin{remark}
If we let $\bar \Sigma^2$ denote the covariance matrix of the above Theorem in the case of the VSRW, the corresponding covariance matrix of the random walk $X$ with speed measure $\theta^\om$ is given by $\Sigma^2=\E\left[\theta^\om(0)\right]^{-1}\bar\Sigma^2$ -- see \cite[Remark~1.5]{invariance}.
\end{remark}

\begin{assumption}
\label{ass:limsups}
There exist $p,q,r\in(1,\infty]$ satisfying
\begin{equation}
\label{eq:pqr_eqn}
\frac{1}{r}+\frac{1}{p}\frac{r-1}{r}+\frac{1}{q}<\frac{2}{d}
\end{equation}
such that
\begin{align} \label{eq:cond_pqr}
\mathbb{E}\left[\Big( \frac{\mu^\om(0)}{\theta^\om(0)}\Big)^{p} \, \theta^\om(0) \right]+\mathbb{E}\left[\nu^\om(0)^{q}\right]
+\mathbb{E}\left[\theta^\om(0)^{-1}\right]+\mathbb{E}\left[\theta^\om(0)^{r}\right] \; < \; \infty.
\end{align}
\end{assumption}
While under Assumptions~\ref{ass:ergodic dynamic} and \ref{ass:limsups} Gaussian-type upper bounds on the heat kernel $p_\theta$ have been obtained in \cite{ADS19}, in the present paper our focus is on local limit theorems.  A local limit theorem constitutes a scaling limit of the heat kernel towards the normalized Gaussian transition density of the Brownian motion with covariance matrix $\Sigma^2$, which appears as the limit process in the QFCLT in Theorem~\ref{thm:qfclt general speed}. The Gaussian heat kernel associated with that process will be denoted
\begin{align} \label{eq:def_kt}
k_t(x)\equiv k_t^\Sigma (x):=\frac{1}{\sqrt{(2\pi t)^d\det \Sigma^2}}\exp\Big(-x\cdot(\Sigma^2)^{-1}x/(2t)\Big).
\end{align}

Our first main result is the following local limit theorem for the RCM under general speed measure. For $x\in\bR^d$ write $\floor{x}=(\floor{x_1},...,\floor{x_d})\in\bZ^d$.

\begin{theorem}[Quenched local limit theorem]
\label{thm:qllt general speed}
Let $T_2>T_1>0$, $K>0$ and suppose that Assumptions \ref{ass:ergodicity general speed measure} and \ref{ass:limsups} hold. Then,
$$\lim_{n\to\infty}\sup_{\abs{x}\leq K} \sup_{ t\in [T_1, T_2]} \abs{n^dp_\theta^\omega(n^2t,0,\floor{nx})-ak_t(x)}=0,\quad\text{for }\mathbb{P}\text{-a.e. }\omega,$$
with $a:=\E\big[\theta^\om(0)\big]^{-1}$.
\end{theorem}
\begin{remark} (i) In the case of the CSRW or VSRW Assumption~\ref{ass:limsups} coincides with the moment condition in Theorem~\ref{thm:qfclt general speed}. Indeed, for the CSRW, $\theta^\om\equiv\mu^\om$, choose  $p=\infty$ and relabel $r$ by $p$; for the VSRW, $\theta^\om\equiv1$, choose $r=\infty$. 

(ii) For the sake of a simpler presentation, Theorem~\ref{thm:qllt general speed} is stated for the RCM on  $\bbZ^d$ only. However, its proof extends to RCMs with ergodic conductances satisfying a slightly modified moment condition on a general class of random graphs including  supercritical i.i.d.\ percolation clusters  and clusters in percolation models with long range correlations, see e.g.\ \cite{DRS14, Sa17}.  The corresponding QFCLT has been shown in \cite{DNS18} and a local limit theorem for the VSRW in \cite[Section~5]{ACS20}.

(iii) The quenched local limit theorem has also been established for symmetric diffusions in a stationary and ergodic environment, under analogous assumptions to the above theorem. This is the main result of \cite{CD15}, see Appendix~A therein for a discussion of the general speed case.
\end{remark}

Theorem~\ref{thm:qllt general speed} extends the local limit theorem in \cite[Theorem~1.11]{harnack} for the CSRW to the case of a general speed measure. 
In general, a local limit theorem is a stronger statement than an FCLT. In fact, even in the i.i.d.\ case, where the QFCLT does hold \cite{andres2013invariance}, we see the surprising effect that due to a trapping phenomenon the heat kernel may behave subdiffusively (see \cite{BBHK08}), in particular a local limit theorem may fail. Nevertheless it does hold, for instance, in the case of uniformly elliptic conductances, where $\mathbb{P}(c^{-1} \leq \omega(e) \leq c)=1$ for some $c\geq 1$, or for  random walks on supercritical percolation clusters (see \cite{barlow2009parabolic}). For  sharp conditions on the tails of i.i.d.\ conductances at zero for Harnack inequalities and a local limit theorem to hold we refer to \cite{BKM15}.    Hence, it is clear that some moment condition is necessary. In the case of the CSRW under general ergodic conductances the moment condition in Assumption~\ref{ass:limsups} is known to be optimal, see \cite[Theorem~5.4]{harnack}.
Further, for the VSRW a quenched local limit theorem has very recently been shown in \cite{BS20a} under  the weaker moment condition with $1/p+1/q<2/(d-1)$.
Local limit theorems have also been obtained in slightly different settings, see  \cite{croydon2008local}, where some general criteria for local limit theorems have been provided  based on the arguments in \cite{barlow2009parabolic}. Finally, stronger quantitative homogenization results for heat kernels and Green functions can be established by using techniques from quantitative stochastic homogenization, see \cite[Chapters 8--9]{AKM19} for details in the uniformly elliptic case. This technique has been adapted to the VSRW on percolation clusters in \cite{DG19}, and it is expected that it also applies to other degenerate models.

The proof of the local limit theorem has two main ingredients, the QFCLT in Theorem~\ref{thm:qfclt general speed} and a H\"older regularity estimate for the heat kernel.  For the latter it is common to use a purely analytic approach and to interpret the heat kernel as a fundamental solution of the heat equation $(\partial_t - \cL^\om_\theta) u = 0$.
Here we will follow the arguments in \cite{ACS20} based on De Giorgi's iteration technique. This approach to show H\"older regularity directly circumvents the need for a parabolic Harnack inequality, in contrast to the proofs in \cite{harnack, barlow2009parabolic}, which makes it significantly simpler.  As a by-product to our argument we do obtain a weak parabolic Harnack inequality (Proposition~\ref{prop:weak harnack}) and a lower near-diagonal heat kernel estimate (Corollary~\ref{cor:near diag ests}). In  \cite[Theorem~3]{DG19},  following again the approach in \cite{AKM19},  stronger Lipschitz continuity of the heat kernel on i.i.d.\ percolation clusters has been shown, which matches the gradient of the Gaussian heat kernel. 

Applications of homogenisation results such as FCLTs and local limit theorems in statistical mechanics often require convergence under the annealed measure. While a QFCLT does imply an annealed FCLT in general, the same does not apply to the local limit theorem.
Next we provide an annealed local limit theorem under a stronger moment condition, which we do not expect to be optimal.

\begin{theorem}[Annealed local limit theorem]
\label{thm:annealed llt general speed}
Suppose Assumption~\ref{ass:ergodicity general speed measure} holds. There exist exponents $p,q,r_1,r_2 \in (1,\infty)$ (only depending on $d$) such that if
\begin{align*}
\mathbb{E}\big[ \mu^\om(0)^{p}\big]+\mathbb{E}\big[\nu^\om(0)^{q}\big]
+\mathbb{E}\big[\theta^\om(0)^{-r_1}\big]+\mathbb{E}\big[\theta^\om(0)^{r_2}\big] \; < \; \infty
\end{align*}
then the following holds. For all $K>0$ and $0<T_1\leq T_2$,
\begin{equation}
\label{eq:annealed llt general speed}
\lim_{n\to\infty}\E\bigg[\sup_{\abs{x}\leq K} \sup_{t\in [T_1, T_2]}\abs{n^dp_\theta^\omega(n^2t,0,\floor{nx})-ak_t(x)}\bigg]=0.
\end{equation}
\end{theorem}
\begin{remark}
In the case of the VSRW, i.e.\ $\theta^\om\equiv 1$, the moment condition required in Theorem~\ref{thm:annealed llt general speed} is more explicitly given by
$\E\big[\omega(e)^{2(\kappa'\vee p)}\big]  + \E\big[\omega(e)^{-2(\kappa'\vee q)}\big] <  \infty$, $ e\in E_d$,
for some $p,\,q\in(1,\infty)$ such that $1/p+1/q<2/d$ and $\kappa'=\kappa'(d,p,q,\infty)$ defined in Proposition~\ref{prop:general speed maximal L1 2} below. Similarly, in the case of the CSRW, $\theta^\om\equiv\mu^\om$, the condition reduces to
$\E\big[\omega(e)^{4\kappa'\vee 2p}\big] + \E\big[\omega(e)^{-(4\kappa'+2)\vee 2q}\big] \;< \; \infty$, $e\in E_d$,
again for some $p,\,q\in(1,\infty)$ such that $1/p+1/q<2/d$ and $\kappa'=\kappa'(d,\infty,q, p)$ defined as in Proposition~\ref{prop:general speed maximal L1 2}. 
\end{remark}

 As mentioned above, the proofs of the quenched local limit theorems in \cite{harnack} and Theorem~\ref{thm:qllt general speed}  rely on H\"older regularity
estimates on the heat kernel, which involve random constants depending on the exponential of the conductances. Those constants can be controlled almost surely, but naively taking expectations would
require exponential moment conditions  stronger than the polynomial moment conditions in Assumption~\ref{ass:limsups}. To derive the annealed local limit theorem given the corresponding quenched result, one might hope to employ the dominated convergence theorem, which requires that the integrand above can be dominated uniformly in $n$ by an integrable function. We achieve this using a  maximal inequality from \cite{ADS19}. Then it is the form of the random constants in this inequality that allows us to anneal the result using only polynomial moments, together with a simple probabilistic bound. 

\subsection{Main Results on the Dynamic RCM} \label{sec:introdyn}
Next we introduce the dynamic random conductance model.
We endow $G=(\bbZ^d, E_d)$, $d\geq 2$, with a family $\omega=\{\omega_t(e)\in(0,\infty):e\in E_d,\,t\in \bR\}$ of positive, time-dependent weights. For $t\in \bR$, $x\in\bbZ^d$, let
\begin{align*}
\mu_t^\omega(x):=\sum_{y\sim x}\omega_t(x,y), \qquad
\nu_t^\omega(x):=\sum_{y\sim x}\frac{1}{\omega_t(x,y)}.
\end{align*}

We define the \textit{dynamic variable speed random walk} starting in $x\in\bZ^d$ at $s\in\bR$ to be the continuous-time Markov chain $(X_t:t\geq s)$ with time-dependent generator 
$$\left(\mathcal{L}_t^\omega f\right)(x):=\sum_{y\sim x}\omega_t(x,y)\left(f(y)-f(x)\right),$$
acting on bounded functions $f:\,\mathbb{Z}^d\rightarrow \mathbb{R}$. Note that the counting measure, which is time-independent, is an invariant measure for $X$. In contrast to Section~\ref{sec:intro_static}, the results in this subsection, like many results on the dynamic RCM,  are restricted to this specific speed measure.
We denote $P_{s,x}^\omega$ the law of this process started at $x\in\mathbb{Z}^d$ at time $s$, and $E_{s,x}^\omega$ the corresponding expectation. For $x,y\in\mathbb{Z}^d$ and $t\geq s$, we denote $p^\omega(s,t,x,y)$ the heat kernel of $\left(X_t\right)_{t\geq s}$, that is
$$p^\omega(s,t,x,y):=P_{s,x}^\omega\left[X_t=y\right].$$
Let $\Omega$ be the set of measurable functions from $\bbR$ to $(0, \infty)^{E_d}$ equipped with a $\sigma$-algebra $\mathcal{F}$ and let $\prob$ be a probability measure on $(\Omega, \mathcal{F})$.
 Upon it we consider the $d+1$-parameter group of translations $(\tau_{t,x})_{(t,x)\in\mathbb{R}\times\mathbb{Z}^d}$ given by
$$\tau_{t,x}:\Omega\rightarrow \Omega,\quad \left(\omega_s(e)\right)_{s\in\bR,\,e\in E_d}\mapsto \left(\omega_{t+s}(x+e)\right)_{s\in\bR,\,e\in E_d}.$$
The required ergodicity and stationarity assumptions on the time-dependent random environment are as follows.

\begin{assumption}
\label{ass:ergodic dynamic}
\begin{enumerate}[(i)]
\item $\mathbb{P}$ is ergodic with respect to time-space translations, i.e.\ for all $x\in\mathbb{Z}^d$ and $t\in\bR$, $\mathbb{P}\circ \tau_{t,x}^{-1}=\mathbb{P}$. Further, $\mathbb{P}(A)\in\{0,1\} \text{ for any } A\in\mathcal{F} \text{ such that }\tau_{t,x}(A)=A \text{ for all }x\in\mathbb{Z}^d, t\in\bR.$
\item For every $A\in\mathcal{F}$, the mapping $(\omega,t,x)\mapsto\indicator_A(\tau_{t,x}\omega)$ is jointly measurable with respect to the $\sigma$-algebra $\mathcal{F}\otimes\mathcal{B}\left(\mathbb{R}\right)\otimes 2^{\mathbb{Z}^d}$.
\end{enumerate}
\end{assumption}

\begin{theorem}[Quenched FCLT and  local limit theorem] \label{thm:quenched_dynamic}
Suppose Assumption~\ref{ass:ergodic dynamic} holds and there exist $p, q \in (1, \infty]$ satisfying
  \begin{align*}
    \frac{1}{p-1} \,+\, \frac{1}{(p-1) q} \,+\, \frac{1}{q}
    \;<\;
    \frac{2}{d}
  \end{align*}
  such that  $\mean\!\big[\om_0(e)^p\big] < \infty$ and $\mean\!\big[\om_0(e)^{-q}\big] < \infty$  for any $e \in E_d$.  
\begin{enumerate}[(i)]
\item The QFCLT holds with a deterministic non-degenerate covariance matrix $\Si^2$.

\item For any $T_2>T_1>0$ and $K>0$,
$$\lim_{n\to\infty}\sup_{\abs{x}\leq K} \sup_{ t\in [T_1, T_2]} \big| n^dp^\omega(0,n^2t,0,\floor{nx})-k_t(x) \big| \; = \; 0,\quad\text{for }\mathbb{P}\text{-a.e. }\omega,$$
\end{enumerate}
where $k_t$ still denotes the heat kernel of a Brownian motion on $\mathbb{R}^d$ with covariance $\Sigma^2$.
\begin{proof}
The QFLCT in (i) has been proven in \cite{ACDS18}, see \cite{BR18} for a similar result. For the quenched local limit theorem in (ii) we refer to \cite{ACS20}.
\end{proof}
\end{theorem} 

Similarly as in the static case we establish an annealed local limit theorem for the dynamic RCM under a stronger, but still polynomial moment condition.

\begin{theorem}[Annealed local limit theorem]
\label{thm:annealed llt dyn}
Suppose  Assumption \ref{ass:ergodic dynamic} holds. There exist exponents $p,q\in (1,\infty)$ (specified more explicitly in Assumption~\ref{ass:poly moments dynamic} below) such that if $\mathbb{E}\big[\omega_0(e)^{p}\big] <  \infty$ and $\mathbb{E}\big[\omega_0(e)^{-q}\big]< \infty$ for any $e\in E_d$, then the following holds.  For all $K>0$ and $0<T_1\leq T_2$,
\begin{equation}
\label{eq:annealed llt dyn}
\lim_{n\rightarrow \infty}\mathbb{E}\bigg[\sup_{\abs{x}\leq K} \sup_{t\in[T_1,T_2]}\big| n^dp^\omega(0,n^2t,0,\floor*{nx})-k_t(x) \big|\bigg]=0.
\end{equation}
\end{theorem}

An annealed local limit theorem has been stated in the uniformly elliptic case in \cite{andres2014invariance}. We do not expect the moment conditions in Theorem~\ref{thm:annealed llt dyn} to be optimal and that they can be relaxed. In an upcoming paper \cite{DKS20} an annealed local limit theorem is obtained for ergodic  conductances uniformly bounded from below but only having a finite first moment by using an entropy argument from \cite{BDKY15}.

Relevant examples of dynamic RCMs include random walks in an environment generated by some interacting particle systems like zero-range or exclusion processes (cf.\ \cite{MO16}). Some on-diagonal heat kernel upper bounds for a degenerate time-dependent conductance model are obtained in \cite{MO16}. Full two-sided Gaussian bounds have been shown in the uniformly elliptic case for the VSRW \cite{DD} or for the CSRW under effectively non-decreasing conductances \cite{DHZ19}. 
However, unlike for the static environments, two-sided Gaussian-type heat kernel bounds are  much less regular and some pathologies may arise as they are not stable under perturbations, see \cite{HK16}. Moreover, in the degenerate case such bounds are expected to be governed by the intrinsic distance.  Even in the static case, in contrast to the CSRW, the intrinsic distance of the VSRW is not comparable to the Euclidean distance in general, cf.\ \cite{ADS19}, and 
the exact form of a time-dynamic version of the distance is still unknown.   These facts make the derivation of Gaussian bounds for the dynamic RCM with unbounded conductances a subtle open challenge.


\subsection{Application to the Ginzburg-Landau $\nabla\phi$ Model}
 A somewhat unexpected context in which one encounters (dynamic) RCMs is that of gradient Gibbs measures describing stochastic interfaces in statistical mechanical systems. One well-established model is the Ginzburg-Landau model, where an interface is described by a field of height variables $\{\phi_t(x): x\in \mathbb{Z}^d, t\geq 0\}$, whose stochastic dynamics are governed by the following infinite system of stochastic differential equations involving nearest neighbour interaction: 
\begin{align} \label{eq:SDE_intro}
\phi_t(x)=\phi(x)-\int_0^t \sum_{y:|x-y|=1} V'(\phi_t(x)-\phi_t(y)) \, dt + \sqrt{2} \, w_t(x), \qquad x\in \mathbb{Z}^d.
\end{align}
Here $\{w(x): x\in \mathbb{Z}^d\}$ is a collection of independent Brownian motions and the potential $V\in C^2(\mathbb{R},\mathbb{R}_+)$ is even and convex. The formal equilibrium measure for the dynamic is given by the Gibbs measure $Z^{-1} \exp(-H(\phi)) \prod_x d\phi(x)$ on $\mathbb{R}^{\mathbb{Z}^d}$ with  formal Hamiltonian $H(\phi) = \frac{1}{2} \sum_{x\sim y} V(\phi(x)-\phi(y))$. Investigating the fluctuations of the macroscopic interface has been quite an active field of research, see \cite{Fu05} for a survey. 

We are interested in the decay of the space-time covariances of height variables under an equilibrium Gibbs measure. By the  \emph{Helffer-Sj\"{o}strand representation} \cite{HS} (cf.\ also \cite{DD, GOS}) such covariances can be written in terms of the \emph{annealed} heat kernel of a random walk among dynamic random conductances. More precisely, 
\[
\cov_\mu\big(\phi_0(0), \phi_t(y)\big)=\int_0^\infty \mathbb{E}_\mu \left[ p^{\omega}(0,t+s,0,y) \right] ds,
\]
where the covariance and expectation are taken with respect to an ergodic Gibbs measure $\mu$ and $p^\omega$ denotes the heat kernel of the dynamic RCM with time-dependent conductances given by
\begin{align} \label{eq:defOmPhi}
\omega_t(x,y):= V''\big(\phi_t(y)-\phi_t(x)\big), \qquad \{x,y\}\in E_d, \, t\geq 0.
\end{align}
Thus far, applications of the aforementioned Helffer-Sj\"{o}strand relation have mostly been restricted to gradient models with a strictly convex potential function that has second derivative bounded above. This corresponds to uniformly elliptic conductances in the random walk picture. However, recent developments in the setting of degenerate conductances will also allow some potentials that are strictly convex but may have faster than quadratic growth at infinity. As our first main result in this direction, we use the annealed local limit theorem of Theorem~\ref{thm:annealed llt dyn} to derive  a scaling limit for the space-time covariances of the $\phi$-field for a wider class of potentials.

\begin{theorem} \label{thm:cov_lim} 
Suppose $d\geq 3$ and let $V\in C^2(\bR)$ be even with $V''\geq c_->0$. 
Then for all $h\in\bR$ there exists a stationary, shift-invariant, ergodic $\phi$-Gibbs measure $\mu$  of mean $h$, i.e.\ $\E_{\mu}[\phi(x)]=h$ for all $x\,\in\,\bZ^d$.
Further, assume that
\begin{align} \label{eq:moment_gradphi}
\mathbb{E}_{\mu}\Big[V''\big(\phi(y)-\phi(x)\big)^{\overline{p}}\Big] \;< \;  \infty, \qquad \text{for any $\{x,y\}\in E_d$,}
\end{align}
 with $\overline{p}:=(2+d)(1+2/d+\sqrt{1+1/d^2})$. Then for all $t>0$ and $x\in\bR^d$,
$$\lim_{n\to\infty}\,n^{d-2}\cov_{\mu}\!\big(\phi_0(0), \phi_{n^2t}(\floor{nx})\,\big)=\int_0^\infty k_{t+s}(x)\,ds,$$
where $k_t$ is the heat kernel from Theorem~\ref{thm:quenched_dynamic} with conductances as given in \eqref{eq:defOmPhi}.
\end{theorem}
Here $\mathbb{E}_{\mu}$ and $\cov_{\mu}$ denote expectation and covariance w.r.t.\ the law of the process $(\phi_t)_{t\geq 0}$ with initial distribution $\mu$.
Theorem~\ref{thm:cov_lim} extends the scaling limit result of \cite[Theorem 5.2]{andres2014invariance} to hold for potentials $V$ for which $V''$ may be unbounded above.
Note that Theorem~\ref{thm:cov_lim} also contains an existence result for  stationary, shift-invariant, ergodic $\phi$-Gibbs measures whose derivation in the present setting requires some extra consideration. We obtain the existence from  the Brascamp-Lieb inequality together with an existence and uniqueness result for the system of  SDEs \eqref{eq:SDE_intro}, see Proposition~\ref{prop:sde solution}, which in turn can be derived following the arguments in \cite[Chapter~4]{Roy07}. 

Our final main result is a scaling limit for the time-static height variables under the $\phi$-Gibbs measure towards a Gaussian free field (GFF). We refer to  \cite[Theorem~A]{NS97}, \cite[Corollary~2.2]{GOS}, \cite[Theorem~2.4]{BS11}  and \cite[Theorem~9]{NW18} for similar results.
For $f\in C_0^\infty (\bbR^d)$, we denote a rescaled version of this $f_n(x):=n^{-(1+d/2)}f(x/n)$ for $n\in\bbZ_+$. We will consider the field of heights acting as a linear functional on such a test function,
\begin{align} \label{eq:defPhiScaled}
\phi(f_n):=n^{-(1+d/2)} \int_{\bbR^d}f(x) \, \phi(\floor{nx})\,dx.
\end{align}

\begin{theorem}[Scaling to GFF]
\label{thm:equilibrium scaling}
Suppose $d\geq 3$ and let $V\in C^2(\bR)$ be even with $V''\geq c_->0$. Let $\mu$ be a stationary, ergodic $\phi$-Gibbs measure of mean $0$. Assume
\begin{align*} 
\mathbb{E}_{\mu}\Big[V''\big(\phi(y)-\phi(x)\big)^{p}\Big] \;< \;  \infty, \qquad \text{for any $\{x,y\}\in E_d$,}
\end{align*}
for some $p >1+\frac{d}{2}$. Then for any $\lambda \in\bbR$ and $f\in C_0^\infty(\bbR^d)$,
\begin{align*}
\lim_{n\to\infty}\bbE_\mu\Big[\exp(\lambda \phi(f_n))\big]=\exp\Big(\frac{\lambda^2}{2}\int_{\bbR^d}f(x)(Q^{-1}f)(x)\,dx\Big),
\end{align*}
where $Q^{-1}$ is the inverse of $Q f:=\sum_{i,j=1}^d q_{ij}\frac{\partial^2 f}{\partial x_i\,\partial x_j}$ and $q=\Si^2$ the covariance matrix from Theorem~\ref{thm:quenched_dynamic} with conductances given by \eqref{eq:defOmPhi}.
\end{theorem}

\begin{remark}
(i) Note that in \eqref{eq:defPhiScaled} the height variables are scaled by $n^{-(1+d/2)}$ while the conventional scaling for a central limit theorem is $n^{-d/2}$. This stronger scaling is required due to strong correlations of the  height variables (cf.\ \cite{NS97, BS11}), in contrast to the scaling limit of the gradient field, which has weaker correlations and only requires the standard scaling $n^{-d/2}$ (cf.\ \cite{GOS, NW18}).

(ii) Having established Theorem~\ref{thm:cov_lim}, a natural next goal is to study the equilibrium space-time fluctuation of the interface and to derive a stronger, time-dynamic version of Theorem~\ref{thm:equilibrium scaling}. See \cite[Theorem~2.1]{GOS} for the case where the potential additionally has second derivative uniformly bounded above. However, this requires extending Theorem~\ref{thm:cov_lim} from a pointwise result to a scaling limit for the covariances of the $\phi$-field integrated against test functions, cf.\ \cite[Proposition 5.1]{GOS}.
In order to control the arising remainder term, we believe such an extension would require upper off-diagonal heat kernel estimates for the dynamic RCM in a degenerate, ergodic environment, which are not available at present, as discussed at the end of Section~\ref{sec:introdyn}.

\end{remark}

Finally, we provide a verification of the moment assumptions in Theorems~\ref{thm:cov_lim} and \ref{thm:equilibrium scaling} for a class of potentials $V$ with $V''$ having polynomial growth.

\begin{proposition}
\label{prop:moments}
Suppose $d\geq 3$ and let the potential $V\in C^2(\bbR)$ be even, satisfying $V''\geq c_->0$. Let $\mu$ be any ergodic, shift-invariant, stationary $\phi$-Gibbs measure. Then for all $p>0$, $\bbE_{\mu}\big[\abs{\phi_t(x)}^p\big]<\infty$ for any $x\in\bbZ^d$ and $t\geq0$.
\end{proposition}

\begin{example}
The above proposition shows that Theorems~\ref{thm:cov_lim} and~\ref{thm:equilibrium scaling} apply to polynomial potentials of  interest, for example the anharmonic crystal potential $V(x)=x^2+\lambda x^4$ ($\lambda>0$), for which the spatial correlation decay is discussed in \cite{bricmont1981lattice}.
\end{example}

\subsection{Notation}
We finally introduce some further notation used in the paper. We write $c$ to
denote a positive, finite constant which may change on each appearance. Constants denoted by $c_i$ will remain the same.  
We endow the graph $G=(\bbZ^d,E_d)$ with the natural graph distance $d$, i.e.\ $d(x,y)$ is the minimal length of a path between $x$ and $y$. Denote $B(x,r):=\{y\in\mathbb{Z}^d : d(x,y)\leq r\}$ the closed ball with centre $x$ and radius $r$. For a non-empty, finite, connected set $A\subseteq \bZ^d$,  we denote by $\partial A
:= \{x \in A: d(x,y)=1\,\text{ for some } y \in A^c \}$ the inner boundary and by $\partial^+ A
:= \{x \in A^c: d(x,y)=1\,\text{ for some } y \in A \}$ the outer boundary of $A$. We write $\overline{A}=A\cup \partial^+ A$ for the closure of $A$.
The graph is given the counting measure, i.e.\ the measure of $A\subseteq \mathbb{Z}^d$ is the number $\vert A\vert$ of elements in $A$.
For $f\!: \bbZ^d \to \bbR$ we define the operator $\nabla$ by
\begin{align*}
  \nabla f\!: E_d \to \bbR,
  \qquad
  E_d \ni e \;\longmapsto\; \nabla f(e) \;:=\; f(e^+) - f(e^-).
\end{align*}
We denote inner products as follows; for $f,\,g:\bZ^d\to\R$ and a weighting function $\phi:\Z^d\to \R$,
$\langle f,g\rangle_{\ell^2(\Z^d,\phi)}:=\sum_{x\in\Z^d}f(x)g(x)\phi(x)$
and if $f,\,g:\E_d\to\R$,  $\langle f, g\rangle_{\ell^2(E_d)}:=\sum_{e\in E_d}f(e)g(e).$ The corresponding weighted norm is denoted $\norm{f}_{l^2(\bbZ^d,\phi)}$.
The Dirichlet form associated with the operator $\cL_\theta^\omega$ is
$$\mathcal{E}^\omega(f,g):=\langle f,-\mathcal{L}_\theta^\omega g\rangle_{\ell^2(\Z^d,\theta)}\equiv\langle \nabla f,\omega\nabla g\rangle_{\ell^2(E_d)},$$
acting on bounded $f,\, g:\Z^d\rightarrow\R$. We will use the shorthand $\cE^\om(f):=\cE^\om(f,f)$. For non-empty, finite $B\subseteq \mathbb{Z}^d$ and $p\in(0,\infty)$,
space-averaged $\ell^p$-norms on functions $f:B\rightarrow\mathbb{R}$ will be used, 
$$\| f\|_{p,B}:=\bigg(\frac{1}{\vert B\vert}\sum_{x\in B}\vert f(x)\vert ^p\bigg)^{1/p}\quad \text{and}\quad \| f\|_{\infty,B}:=\max_{x\in B}\vert f(x)\vert.$$
Now let $Q=I\times B$ where $I\subseteq \R$ is compact. Let $u:Q\to\R$ and denote $u_t:B\to\R$, $u_t(\cdot):=u(t,\cdot)$ for $t\in I$. For $p'\in(0,\infty)$, we define the space-time  averaged norms  
$$\norm{u}_{p,p',Q}:=\left(\frac{1}{\abs{I}}\int_I\norm{u_t}_{p,B}^{p^\prime}dt\right)^{1/p}\quad\text{and}\quad\norm{u}_{p,\infty,Q}:=\max_{t\in I}\norm{u_t}_{p,B}.$$
Furthermore, we will work with two varieties of weighted norms
\begin{align*}
\norm{f}_{p,B,\phi}&:=\bigg(\frac{1}{\phi(B)}\sum_{x\in B}\abs{f(x)}^p \, \phi(x)\bigg)^{\! 1/p}, \quad \onorm{f}_{p,B,\phi}:=\bigg(\frac{1}{\abs{B}}\sum_{x\in B}\abs{f(x)}^p \, \phi(x)\bigg)^{\!1/p},
\\
\norm{u}_{p,p^\prime,Q,\phi}&:=\left(\frac{1}{\abs{I}}\int_I\norm{u_t}_{p,B,\phi}^{p^\prime}\,dt\right)^{1/p'}, \quad
\norm{u}_{p,\infty,Q,\phi}:=\max_{t\in I}\norm{u_t}_{p,B,\phi},\\
\onorm{u}_{p,p',Q,\phi}&:=\left(\frac{1}{\abs{I}}\int_I\onorm{u_t}_{p,B,\phi}^{p'}\,dt\right)^{1/p'},
\end{align*}
for a weighting function $\phi:B\rightarrow(0,\infty)$, where $\phi(B):=\sum_{x\in B}\phi(x)$.

\subsection{Structure of the Paper}
Section~\ref{section:general speed measure} is devoted to the proof of the quenched local limit theorem for general speed measures - Theorem~\ref{thm:qllt general speed}.  The annealed local limit theorems for the static and dynamic RCM,  Theorem~\ref{thm:annealed llt general speed} and Theorem~\ref{thm:annealed llt dyn}, are shown in Sections~\ref{section:annealed general speed} and \ref{section:dynamic}, respectively. Finally, the application to the Ginzburg-Landau interface model is discussed in Section~\ref{section:interface}.

\section{Local Limit Theorem for the Static RCM under General Speed Measure}
\label{section:general speed measure}
For the proof of Theorem~\ref{thm:qllt general speed} we shall follow a method first developed in \cite{croydon2008local} and \cite{barlow2009parabolic}, for which the main ingredients are the QFCLT in Theorem~\ref{thm:qfclt general speed} and a H\"older regularity estimate for the heat kernel. To derive the latter we adapt the techniques employed in \cite{ACS20} to the general speed measure case.
The key result in Theorem~\ref{thm:oscillations} is an oscillation inequality for solutions of $\partial_t u - \cL^\om_\theta u=0$, such as the heat kernel,  which  implies the required H\"older regularity by a simple iteration argument (see Proposition \ref{prop:cont_hk} below). For the proof of the oscillation inequality, we first derive a maximal inequality (see Theorem \ref{thm:maximal})  using a De Giorgi iteration scheme in Section~\ref{sec:max_ineq}. Then we bound  the sizes of the level sets of a solution $u$ in terms of $(-\ln u)_+$ (see Lemmas~\ref{lemma:u bound space} and \ref{lemma:u bound cylinder} below). These two steps are sufficient to prove the oscillation inequality following an idea in \cite{wu2006elliptic}, see Section~\ref{sec:osc}. To begin with, we collect the required functional inequalities in Section~\ref{sec:func_ineq}.

\subsection{Sobolev and Weighted Local Poincar\'e Inequalities} \label{sec:func_ineq}

One auxiliary result which will prove useful is a modification of the Sobolev inequality derived in \cite{invariance}.

\begin{proposition} \label{prop:sobolev}
Let $d\geq 2$ and $B\subset \bZ^d$ be finite and connected. For any $q\in [1,\infty]$ there exists $c_1= c_1(d,q)$ such that for  any $v:\bbZ^d\to\bbR$ with $v\equiv 0$ on $\partial B$,
\begin{align*}
\norm{v^2}_{\rho,B} \; \leq \;  c_1 \, |B|^{2/d} \, \norm{\nu^\omega}_{q,B} \,  \norm{\theta^\omega}_{1,B} \, \frac{\cE^\omega(v)}{\theta^\omega(B)},
\end{align*}
where $\rho := qd / (q(d-2)+d)$.
\end{proposition}
\begin{proof}
By \cite[equation (28)]{invariance},
\begin{align*}
\norm{ v^2}_{\rho,B} \; \leq \; c_1 \, |B|^{2/d} \, \norm{\nu^\omega}_{q,B}\ \, \frac{\cE^\omega(v)}{\abs{B}},
\end{align*}
and since $\norm{\theta^\omega}_{1,B} = \theta^\om(B)/ \abs{B}$ this gives the claim.
\end{proof}

Another input is a weighted Poincar\'e inequality which will be applied in deriving the aforementioned oscillations bound. We denote the weighted average of any $u:\bbZ^d \rightarrow \bbR$ over a finite subset $B\subset \bbZ^d$ with respect to some $\phi : \bbZ^d \to \bbR $,
$$(u)_{B,\phi}\; := \;\frac{1}{\phi(B)} \sum_{x\in B} u(x) \, \phi(x).$$
We shall also write $(u)_B:=(u)_{B,1}$ when $\phi\equiv 1$.

\begin{proposition}
\label{prop:poincare}
Let $d\geq 2$. There exists $c_{2}=c_{2}(d)<\infty$ such that for any ball $B(n):=B(x_0,n)$ with $x_0\in \bbZ^d$ and $n\geq 1$, any non-empty $\mathcal{N}\subseteq B$ and $u:\bbZ^d\rightarrow\R$,
 \begin{align}\label{eq:local:PI:weighted}
    &\Norm{u - (u)_{B(n), \theta}}{1, B(n), \theta}^2
   \; \leq\;
    c_{2}\,  \cA^\om_1(n) \;
    \frac{n^2}{|B(n)|}
    \sum_{\substack{x, y \in B(n) \\ x \sim y}} \mspace{-8mu}
    \om(x, y)\, \big( u(x) - u(y) \big)^2,
  \end{align}
and
\begin{align} \label{eq:local:PI:weighted_sets} 
& \norm{u-\left(u\right)_{\mathcal{N},\theta}}_{1,B(n),\theta}^{2} \nonumber\\[.5ex]
    &\mspace{36mu} \leq \;  c_2 \,   \cA^\om_1(n) \, \bigg(1+\frac{\theta^\om(B(n))}{\theta^\om(\mathcal{N})}\bigg)^{\!2}  
 \frac {n^2} {\abs{B(n)}} \sum_{\substack{x,y\in B(n)\\ x\sim y}}\omega(x,y) \, \big(u(x)-u(y)\big)^2
\end{align}
with $\cA_1^\om(n) :=  \norm{1/\theta^\om}_{1,B(n)}^{2} \, \norm{\theta^\om}_{r,B(n)}^2 \, \norm{\nu^\omega}_{q,B(n)}$.
\end{proposition}
\begin{proof}
By a discrete version of the co-area formula the classical local $\ell^1$-Poincar\'e inequality on $\bbZ^d$ can be easily established, see e.g.\ \cite[Lemma~3.3.3]{SC97}, which also implies an $\ell^{\al}$-Poincar\'e inequality for any $\al \in [1, d)$.   Note that, by \cite[Th\'eor\`eme~4.1]{Cou96}, the volume regularity of balls and the local $\ell^{\al}$-Poincar\'e inequality on $\bbZ^d$ implies that for $d \geq 2$ and any $u\!: \bbZ^d \to \bbR$,
\begin{align}\label{eq:Sobolev:Sa}
  \inf_{a \in \bbR}
  \Norm{u - a}{\frac{d \al}{d-\al}, B(n)}
  \;\leq\;
  c \, n\,
  \bigg(
    \frac{1}{|B(n)|}\,
    \sum_{\substack{x, y \in B(n) \\ x \sim y}} \mspace{-8mu}
    \big| u(x) - u(y) \big|^\al
  \bigg)^{\!\!1/\al}.
\end{align}
Further, for any $\al \in [1, 2)$, H\"older's inequality yields
  \begin{align}\label{eq:PI:rhs}
    &\bigg(
      \frac{1}{|B(n)|} \!\!
      \sum_{\substack{x, y \in B(n)\\ x \sim y}} \mspace{-10mu} \big|u(x) - u(y)\big|^{\al}
    \bigg)^{\!\!\frac 1 \al}
    \leq
    \Norm{\nu^{\om}}{\frac{\al}{2 - \al}}^{1 / 2}
    \bigg(
      \frac{1}{|B(n)|} \!
      \sum_{\substack{x, y \in B(n)\\ x \sim y}} \mspace{-10mu}
        \om(x, y)  \big( u(x) - u(y) \big)^2
    \bigg)^{\!\! \frac 1 2}.
  \end{align}
Note that by \cite[Lemma 2]{DK13}, we have for any $a\in\bbR$,
$$\Norm{u - (u)_{B(n), \theta}}{1, B(n), \theta}\leq c\, \Norm{u - a}{1, B(n), \theta}.$$

  Now we prove \eqref{eq:local:PI:weighted} by distinguishing two cases. In the case $r\geq 2$ we have by Cauchy-Schwarz,
\begin{align*} 
\norm{u-a}_{1,B(n),\theta} \; \leq \; \norm{\theta^\om}_{1,B(n)}^{-1} \, \norm{\theta^\om}_{2,B(n)} \, \norm{u-a}_{2,B(n)}.
\end{align*}

Hence we obtain the assertion \eqref{eq:local:PI:weighted} by using \eqref{eq:Sobolev:Sa} and \eqref{eq:PI:rhs} with the choice $\al = 2d / (d+2)$ and Jensen's inequality. Note that $\alpha/(2-\alpha)=d/2<q$.

Similarly, in the case $r\in[1,2)$, denoting its H\"older conjugate $r_*$ we have by H\"older's inequality
\begin{align*} 
\norm{u-a}_{1,B(n),\theta} \; \leq \; \norm{\theta^\om}_{1,B(n)}^{-1} \, \norm{\theta^\om}_{r,B(n)} \, \norm{u-a}_{r_*,B(n)},
\end{align*} 
and we may use \eqref{eq:Sobolev:Sa} and \eqref{eq:PI:rhs} with the choice $\al = dr_* / (d+r_*)$. Notice that $d\alpha/(d-\alpha)=r_*$, $\alpha/(2-\alpha)\leq q$ and $\alpha \in [1,2)$ since $r\in [1,d]$ and satisfies \eqref{eq:cond_pqr}. This finishes the proof of \eqref{eq:local:PI:weighted} .

To see \eqref{eq:local:PI:weighted_sets}, note that by the triangle inequality
\begin{align*}
& \norm{u-\left(u\right)_{\mathcal{N},\theta}}_{1,B(n),\theta}
\;\leq \;
\norm{u-(u)_{B(n),\theta}}_{1,B(n),\theta} \, + \, \big| (u)_{\cN,\theta}-(u)_{B(n),\theta} \big|  \nonumber \\
& \mspace{36mu} \;\leq \;
\norm{u-(u)_{B(n),\theta}}_{1,B(n),\theta} \, + \, \frac{1}{\theta^\om(\cN)}\sum_{y\in\cN}\big| u(y)-(u)_{B(n),\theta} \big| \, \theta^\om(y) \nonumber
\\
& \mspace{36mu} \;\leq \; \bigg(1+\frac{\theta^\om(B(n))}{\theta^\om(\mathcal{N})}\bigg)\, \norm{u-(u)_{B(n),\theta}}_{1,B(n),\theta},
\end{align*}
so \eqref{eq:local:PI:weighted_sets} follows from \eqref{eq:local:PI:weighted}.
\end{proof}

\subsection{Maximal Inequality} \label{sec:max_ineq}
For the rest of Section~\ref{section:general speed measure} we assume $d\geq 2$ and we fix $p,q,r\in(1,\infty]$ such that 
\begin{align} \label{eq:pqr2}
\frac{1}{r}+\frac{1}{p}\frac{r-1}{r}+\frac{1}{q} \; < \; \frac{2}{d}.
\end{align}
For the analysis, we work with space-time cylinders defined as follows. For any $x_0\in\bZ^d$ and $t_0\in\R$ let $I_\tau:=[t_0-\tau n^2,t_0]$ and $B_\sigma:=B(x_0,\sigma n)$ for $\sigma\in(0,1],\,\tau\in(0,1]$. We write $Q(n):=[t_0-n^2,t_0]\times B(x_0,n)$ and
$$Q_{\tau,\sigma}(n):=I_\tau\times B_\sigma\quad\text{and}\quad Q_\sigma :=Q_\sigma(n):=Q_{\sigma,\sigma}(n).$$

The main result in this subsection is the following maximal inequality.

\begin{theorem}
\label{thm:maximal}
Let $t_0\in \R$, $x_0\in \bZ^d$ and $u>0$ be such that $\partial_tu-\mathcal{L}^\omega_\theta u\,\leq\,0$ on $Q(n)$ for any $n\geq 1$. Then, for any $0 \leq \Delta <2/(d+2)$
there exists $N_1=N_1(\Delta)\in\bbN$ and $c_3=c_3(d,p,q,r)$ such that for all $n\geq N_1$, $h\geq 0$ and $1/2\leq \sigma' <\sigma \leq 1$ with $\sigma-\sigma' > n^{-\Delta}$,
\begin{align*}
\max_{(t,x) \in Q_{\sigma'}(n)}u(t,x) \;\leq \; h \,+ \, c_3 \, \bigg( \frac{\cA_2^\om(n)}{(\sigma-\sigma')^2} \bigg)^{\!\kappa} \, \norm{(u-h)_+}_{2p_*,2,Q_\sigma(n),\theta}.
\end{align*}
Here $p_*:=p/(p-1)$, $\kappa:=1+p_* \rho/2(\rho - p_* r_*)$ with $\rho$ as in Proposition~\ref{prop:sobolev}, and
\begin{align} \label{eq:defA1}
\cA_2^\om(n) \; :=  \;\norm{1\vee (\mu^\omega/\theta^\om)}_{p,B(n),\theta} \, \norm{1\vee \nu^\omega}_{q,B(n)} \,   \norm{1 \vee \theta^\om}_{r,B(n)}^2 \norm{1\vee (1/\theta^\om)}_{1,B(n)}.
\end{align}
\end{theorem}

An energy estimate is required in proving the above, cf.\ \cite[Lemma 3.7]{ADS19}.

\begin{lemma}
\label{lemma:aux1}
Suppose $Q=I\times B$ where $I=[s_1,s_2]\subseteq \R$ is an interval and $B\subset \bZ^d$ is finite and connected. Let $u$ be a non-negative solution of $\partial_t u-\cL_\theta^\omega u\leq 0$ on $Q$. Let $\eta:\bZ^d\to[0,1]$ and $\xi:\bR\to[0,1]$ be cutoff functions such that $\supp \eta\subseteq B$, $\supp\xi\subseteq I$ and $\eta\equiv 0$ on $\partial B$, $\xi(s_1)=0$. Then there exists $c_4$ such that for any $k\geq 0$ and $p,p_*\in(1,\infty)$ with $\frac{1}{p}+\frac{1}{p_*}=1$,
\begin{align} \label{eq:energy:est}
& \frac{1}{\abs{I}}\norm{\xi\eta^2 (u-k)^2_+}_{1,\infty,Q,\theta}+\frac{1}{\abs{I}}\int_I\xi(t) \, \frac{\mathcal{E}^\omega(\eta v)}{\theta^\om(B)} \, dt \nonumber  \\
& \mspace{36mu}\leq c_4\left(\norm{\mu^\omega/\theta^\om}_{p,B,\theta}\norm{\nabla\eta}_{l^\infty(E_d)}^2+\norm{\xi^\prime}_{L^\infty(I)}\right)\norm{(u-k)^2_+}_{p_*,1,Q,\theta}.
\end{align}
\end{lemma}
\begin{proof}
This follows by a simple modification of the argument in \cite[Lemma 2.9]{ACS20}.
\end{proof}

\begin{proof}[Proof of Theorem \ref{thm:maximal}]
The proof is based on an iteration argument and will be divided into two steps. First we will derive the estimate needed for a single iteration step, then the actual iteration will be carried out. Set $\alpha:=1+\frac{1}{p_*}-\frac{r_*}{\rho}$ with $r_*:=r/(r-1)$. Notice that for any $p,q,r \in (1,\infty]$ satisfying \eqref{eq:pqr2}, $\alpha>1$ and therefore $1/\alpha_*:= 1- 1/\alpha>0$.

\emph{Step 1:} 
Let $1/2\leq \sigma'<\sigma\leq 1$ and $0\leq k<l$ be fixed.
Note that, due to the discrete structure of the underlying space $\bbZ^d$, the balls $B_\sigma$ and $B_{\sigma'}$ may coincide. To ensure that $B_{\sigma'} \subsetneq B_\sigma$ we assume in this step that $(\sigma-\sigma')n\geq 1$. Then, it is possible to define a spatial cut-off function $\eta:\bZ^d\rightarrow[0,1]$ such that $\supp\eta\subseteq B_\sigma, \,\eta\equiv1$ on $B_{\sigma^\prime}$, $\eta\equiv 0$ on $\partial B_\sigma$ and $\norm{\nabla\eta}_{l^\infty(E)}\leq 1/((\sigma-\sigma') n)$. Further, let $\xi \in C^\infty(\R)$ be a cut-off in time satisfying $\supp \xi\subseteq I_\sigma$, $\xi\equiv 1$ on $I_{\sigma^\prime}$, $\xi(t_0-\sigma n^2)=0$ and $\norm{\xi^\prime}_{L^\infty([0,\infty))}\leq 1/((\sigma-\sigma')n^2)$.
By H\"older's inequality, followed by applications of H\"older's and Young's inequalities,
\begin{align} \label{eq:maximal1}
&\norm{\big(u-l\big)_+^2}_{p_*,1,Q_{\sigma^\prime},\theta} \; \leq \; \norm{\big(u-k\big)_+^2}_{\alpha p_*,\alpha,Q_{\sigma^\prime},\theta}\, \norm{\indicator_{\{u\geq l\}}}_{\alpha_* p_*,\alpha_*,Q_{\sigma^\prime},\theta} \nonumber \\
& \mspace{36mu} \; \leq \; \bigg(\norm{\big(u-k\big)_+^2}_{1,\infty,Q_{\sigma^\prime},\theta}\, + \, \norm{(u-k)_+^2}_{\rho/r_*,1,Q_{\sigma^\prime},\theta}\bigg) \, \norm{\indicator_{\{u\geq l\}}}_{p_*,1,Q_{\sigma^\prime},\theta}^{1/\alpha_*}.
\end{align}
Note that by Jensen's inequality
\begin{align} \label{eq:theta_balls}
\frac{\theta^\omega(B_\sigma)}{\theta^\omega(B_{\sigma'})} \; \leq \; c \, \norm{\theta^\omega}_{1,B_\sigma} \, \norm{1/\theta^\omega}_{1,B_{\sigma'}}.
\end{align}
We use H\"older's inequality, the Sobolev inequality in Proposition~\ref{prop:sobolev}, 
 the fact that $r_*/\rho<1$ and Lemma~\ref{lemma:aux1} to obtain
\begin{align} \label{eq:maximal2}
& \norm{\big(u-k\big)_+^2}_{\rho/r_*,1,Q_{\sigma^\prime},\theta}
\; \leq  \; c \, \Big( \norm{\theta^\omega}_{1,B_\sigma} 
 \, \norm{1/\theta^\omega}_{1,B_{\sigma'}}\Big)^{\!\frac{r_*}{\rho}} \,
\norm{\xi \eta^2 \big(u-k\big)_+^2}_{\rho/r_*,1,Q_{\sigma},\theta} \nonumber \\
& \mspace{36mu}
\;\leq \; 
c \, n^2 \, \norm{\nu^\omega}_{q,B_\sigma} \, \Big( \norm{\theta^\omega}_{r,B_\sigma}^2 \,  \norm{1/\theta^\omega}_{1,B_{\sigma}}  \Big)^{\frac{r_*}{\rho}} \, \frac{1}{\abs{I_\sigma}} \int_{I_\sigma}\xi(t)\, \frac{\mathcal{E}^\omega\big(\eta\, (u_t-k)_+\big)}{\theta^\om(B_\sigma)} \, dt \nonumber
\\
& \mspace{36mu} \; \leq \; c \, \frac{\tilde{\cA}_2^\om(n) }{(\sigma-\sigma^\prime)^2} \, \norm{(u-k)_+^2}_{p_*,1,Q_\sigma,\theta},
\end{align}
with $\tilde{\cA}_2^\omega(n):=\cA_2^\omega(n)\,/\, \, \norm{1\vee (1/\theta^\om)}_{1,B_\sigma}$.  Further, again by \eqref{eq:theta_balls} and  Lemma \ref{lemma:aux1},
\begin{align} \label{eq:maximal3}
& \norm{\big(u-k\big)_+^2}_{1,\infty,Q_{\sigma^\prime},\theta}
\; \leq \;
 c \, \norm{\theta^\omega}_{1,B_\sigma} 
 \, \norm{1/\theta^\omega}_{1,B_{\sigma'}} \norm{\xi\eta^2\big (u-k \big)_+^2}_{1,\infty,Q_\sigma,\theta} \nonumber \\
& \mspace{36mu} \;\leq \; 
c \,  \frac{\norm{1\vee (\mu^\omega/\theta^\om)}_{p,B_\sigma,\theta} \,  \norm{\theta^\omega}_{1,B_\sigma} 
 \, \norm{1/\theta^\omega}_{1,B_{\sigma}}}{(\sigma-\sigma^\prime)^2}\norm{(u-k)_+^2}_{p_*,1,Q_\sigma,\theta} \nonumber \\
& \mspace{36mu} \;\leq \; c\, \frac{\tilde{\cA}_2^\omega(n)}{(\sigma-\sigma^\prime)^2} \, \norm{\big(u-k\big)_+^2}_{p_*,1,Q_\sigma,\theta}.
\end{align}
Moreover, note that
\begin{align}
\label{eq:maximal4}
& \norm{\indicator_{\{u\geq l\}}}_{p_*,1,Q_{\sigma^\prime},\theta} \;\leq \; 
 c  \, \norm{\theta^\omega}_{1,B_\sigma} 
 \, \norm{1/\theta^\omega}_{1,B_{\sigma'}} \norm{\indicator_{\{u-k\geq l-k \}}}_{p_*,1,Q_\sigma,\theta} \nonumber \\
& \mspace{36mu}  \;\leq \; \frac{\tilde{\cA}_2^\om(n)}{(l-k)^2} \, \norm{\big(u-k\big)_+^2}_{p_*,1,Q_\sigma,\theta}.
\end{align}
Therefore, combining \eqref{eq:maximal1} with \eqref{eq:maximal2}, \eqref{eq:maximal3} and \eqref{eq:maximal4} yields
\begin{align*}
\norm{\big(u-l\big)_+^2}_{p_*,1,Q_{\sigma^\prime},\theta} \; \leq \;  \frac{ c\, \tilde{\cA}_2^\om(n)^{1+\frac{1}{\alpha_*}}}{ (l-k)^{2/\alpha_*} (\sigma-\sigma^\prime)^2} \, \norm{\big(u-k\big)_+^2}_{p_*,1,Q_\sigma,\theta}^{1+\frac{1}{\alpha_*}}.
\end{align*}
Introducing $\varphi(l,\sigma^\prime):=\norm{\big(u-l\big)_+^2}_{p_*,1,Q_{\sigma^\prime},\theta}$ and setting $M:=c\, \tilde{\cA}_2^\om(n)^{1+\frac{1}{\alpha_*}}$ the above inequality reads
\begin{align} \label{eq:varphi_singlestep}
\varphi(l,\sigma^\prime) \; \leq \; \frac{M}{ (l-k)^{2/\alpha_*} (\sigma-\sigma^\prime)^2} \, \varphi(k,\sigma)^{1+\frac{1}{\alpha_*}}
\end{align}
and holds for any $0\leq k<l$ and $1/2\leq\sigma^\prime<\sigma\leq 1$.

\emph{Step 2:} For any $\Delta \in [0, 2/(d+2))$ let $n\geq N_2(\Delta)$ where $N_2(\Delta)<\infty$ is such that $n^{2/(d+2)-\Delta}\geq 2$ for all $n\geq N_2$. Let $h\geq 0$ be arbitrary and $1/2\leq \sigma'< \sigma\leq 1$ be chosen in such a way that $\sigma-\sigma' > n^{-\Delta}$. Further, for $j\in \bbN$ we set
\begin{align*}
\sigma_j \; := \; 2^{-j} (\sigma-\sigma'), \qquad k_j\; :=\; h \,+ \, K \, \big( 1-2^{-j} \big),
\end{align*}
where $K:= 2^{2(1+\alpha_*)^2} \big(M/(\sigma-\sigma')^2\big)^{\alpha_*/2} \varphi(h,\sigma)^{1/2}$, and  $J:=\lfloor d\ln n/ 2\alpha_* \ln 2 \rfloor$.
Since $\alpha_* \geq (d+2)/2$, we have
\begin{align*}
(\sigma_{j-1} - \sigma_j) n \; = \; 2^{-j} (\sigma - \sigma')n \; > \; 1, \qquad \forall j=1, \ldots, J.
\end{align*}
Next we claim that, by induction,
\begin{align} \label{eq:bound_varphi}
\varphi(k_j,\sigma_j) \; \leq \; \frac{\varphi(h,\sigma)}{r^j}, \qquad \forall j=1,\ldots, J,
\end{align}
where $r=2^{4(1+\alpha_*)}$.
Indeed for $j=0$ the bound \eqref{eq:bound_varphi} is trivial. Now assuming that  \eqref{eq:bound_varphi} holds for any $j-1\in\{0, \ldots, J-1\}$, we obtain by \eqref{eq:varphi_singlestep} that
\begin{align*}
\varphi(k_j,\sigma_j) & \; \leq \; M \, \bigg( \frac{2^j}{K} \bigg)^{2/\alpha_*} \bigg(\frac{2^j}{(\sigma-\sigma^\prime)} \bigg)^2 \, \varphi(k_{j-1},\sigma_{j-1})^{1+\frac{1}{\alpha_*}} \\
& \; \leq \;
 M \, \bigg( \frac{2^j}{K} \bigg)^{2/\alpha_*} \bigg(\frac{2^j}{(\sigma-\sigma^\prime)} \bigg)^2  \, \bigg( \frac{\varphi(h,\sigma)}{r^{j-1}} \bigg)^{1+\frac{1}{\alpha_*}} \; \leq \; \frac{\varphi(h,\sigma)}{r^j},
\end{align*}
which completes the proof of \eqref{eq:bound_varphi}. Note that by the choice of $J$, $(n^{2d}2^{2J})/r^J \leq 1$ and $(\sigma_J-\sigma_{J+1}) n \geq 1$.

By using the Cauchy-Schwarz inequality, \eqref{eq:maximal3} and \eqref{eq:bound_varphi}, we have that
\begin{align*}
&\max_{(t,x) \in Q_{\sigma_{J+1}}} \big(u(t,x)-k_{J+1} \big)_+ 
 \; \leq \;
c\, n^d \, \norm{1/\theta^\om}_{1,B_\sigma}^{1/2} \,  \norm{\big(u-k_{J+1}\big)_+^2}_{1,\infty,Q_{\sigma_{J+1}},\theta}^{1/2} \\
& \mspace{36mu} \; \leq \;
c\, \norm{1/\theta^\om}_{1,B_\sigma}^{1/2}  \bigg( n^{2d} \, 2^{2J} \,  \frac{\tilde{\mathcal{A}}_2^\omega(n)}{(\sigma-\sigma^\prime)^2} \, \varphi(k_J,\sigma_J) \bigg)^{\!1/2} 
 \; \leq \;
c\, \bigg( \frac{\mathcal{A}_2^\omega(n)}{(\sigma-\sigma^\prime)^2} \, \varphi(h,\sigma) \bigg)^{\!1/2} \\
& \mspace{36mu}
\; = \;
c\, \bigg( \frac{\mathcal{A}_2^\omega(n)}{(\sigma-\sigma^\prime)^2}  \bigg)^{\!1/2}  \, \norm{(u-h)_+}_{2p_*,2,Q_\sigma(n),\theta}.
\end{align*}
Hence,
\begin{align*}
\max_{(t,x) \in Q_{\sigma'}} u(t,x) \; \leq \; h \,+ \, K \,+ \,  
c\, \bigg( \frac{\mathcal{A}_2^\omega(n)}{(\sigma-\sigma^\prime)^2}  \bigg)^{\!1/2}  \, \norm{(u-h)_+}_{2p_*,2,Q_\sigma(n),\theta},
\end{align*}
and the claim follows with $\kappa=(1+\alpha_*)/2$ as in the statement.
\end{proof}

\subsection{Oscillation inequality} \label{sec:osc}

The next significant result allows us to control the oscillations of a space-time harmonic function. We denote the oscillation of a function $u$ on a cylinder $Q\subseteq \bbR\times\bZ^d$, $\osc_Q u:=\max_{(t,x)\in Q}u(t,x)-\min_{(t,x)\in Q}u(t,x)$. Recall the definition of $\cA_1^\om(n)$ and  $\cA_2^\om(n)$ in Proposition~\ref{prop:poincare} and \eqref{eq:defA1}, respectively. For $n\geq 4$  we also set  $\cA_3^\om(n):= \norm{1/\theta^\om}_{1,B(\frac n 4)} \, \norm{\theta^\om}_{1,B(\frac  n 2)}$.

\begin{theorem}[Oscillation inequality]
\label{thm:oscillations}
Fix $t_0\in\R,\,x_0\in \bZ^d$. Let $u:\bbZ^d\to\bbR$ be such that $\partial_tu-\mathcal{L}_\theta^\omega u=0$ on $Q(n)$ for $n\geq 1$. There exists $N_3= N_3(d)$ (independent of $x_0$) such that for all $n\geq N_3$ the following holds. There exists 
\begin{align*}
\gamma^\om(x_0,n)=\gamma\big(\cA_1^\om(n), \cA_2^\om(n), \cA_3^\om(n),\norm{\mu^\omega}_{1,B(n)}, \norm{\theta^\om}_{1,B(n)}, \norm{1/ \theta^\om}_{1,B(n)}\big) \in (0,1),
\end{align*}
which is continuous and increasing in all components, 
such that
$$\osc_{Q( n/4)} u \; \leq  \; \gamma^\om(x_0,n) \, \osc_{Q(n)} u.$$
\end{theorem}
Before we prove Theorem~\ref{thm:oscillations} we briefly record the following continuity statement for space-time harmonic functions as one of its consequences.
\begin{corollary} \label{cor:cont_caloric}
Suppose that Assumptions~\ref{ass:ergodicity general speed measure} and \ref{ass:limsups} hold.  Let $\de > 0$, $x_0 \in \bbZ^d$ and $\sqrt{t_0}/2 > \de$ be fixed. Suppose $\partial_t u - \cL_t^{\om} u = 0$  on $[0, t_0] \times B(x_0, n)$. For $\prob$-a.e.\ $\om$, there exist  $N_4=N_4(x_0,\om)$ and $\bar{\ga}\in (0,1)$ (only depending on the law of $\om$ and $\theta^\om$) such that 
if $\de n \geq N_4$,   then for any $t\in n^2 [t_0 - \de^2, t_0]$ and $x_1, x_2 \in B(x_0, \de n)$,
  \begin{align*}
    \big|u(t,x_1) - u(t,x_2)\big|
    \leq
    c_5\, \bigg( \frac{\de}{\sqrt{t_0}} \bigg)^{\!\!\vr}\,
    \max_{[3 t_0 / 4, t_0] \times B(x_0, \sqrt{t_0} / 2)} u,
  \end{align*}
  where $\vr := \ln \bar{\ga} / \ln (1/4)$ and $c_5$ depends only on $\bar{\ga}$.
\end{corollary}
\begin{proof}
This follows from Theorem~\ref{thm:oscillations}  as in \cite[Corollary~2.6]{ACS20}, see also Proposition~\ref{prop:cont_hk} below for a similar proof.
\end{proof}

In the remainder of this subsection we will prove Theorem~\ref{thm:oscillations}
 by following the method in \cite{ACS20}, originally used in \cite{wu2006elliptic} for parabolic equations in continuous spaces. Consider the function $g:(0,\infty)\rightarrow[0,\infty)$, which may be regarded as a continuously differentiable version of the function $x\mapsto(-\ln x)_+$, defined by 
\begin{align*}
g(z):=\threepartdef {-\ln z} {z\in(0,\bar c],} {\frac{(z-1)^2}{2 \bar c (1-\bar c)}}{z\in (\bar c,1],}{0}{z\in(1,\infty),}
\end{align*}
where $\bar c\in[\frac{1}{4},\frac{1}{3}]$ is the smallest solution of the equation $2c\ln(1/c)=1-c$. Note that $g\in C^1(0,\infty)$ is convex and non-increasing. Although $g(u)$ is not space-time harmonic, we can still bound its Dirichlet energy as follows.

\begin{lemma}
\label{lemma:energy estimate}
Suppose $u>0$ satisfies $\partial_t u-\mathcal{L}_\theta^\omega u = 0$ on $Q=I\times B$ with $I$ and $B$ as in Lemma~\ref{lemma:aux1}. Let $\eta: \bZ^d\rightarrow[0,1]$ be a cut-off function with
$\supp\eta\subseteq B$ and $\eta\equiv0$ on $\partial B$. Then,
\begin{equation}
\label{eq:energy est}
\partial_t\norm{\eta^2g(u_t)}_{1,B,\theta}+\frac{\mathcal{E}^{\omega,\eta^2}(g(u_t))}{6 \, \theta^\om(B)} \; \leq  \; 6 \, \frac{\norm{1\vee \mu^\omega}_{1,B}}{\norm{\theta^\om}_{1,B}} \,\osr(\eta)^2 \, \norm{\nabla \eta}_{l^\infty(E_d)}^2,
\end{equation}
where $\osr(\eta):=\max\{(\eta(y)/\eta(x))\vee 1\,|\, \{x,y\}\in E_d,\,\eta(x)\neq 0\}$ and
$$\mathcal{E}^{\omega,\eta^2}(f) \; := \; \sum_{e\in E_d } \big(\eta^2(e^+)\wedge\eta^2(e^-)\big) \, \omega(e)\, (\nabla f)^2(e).$$
\end{lemma}
\begin{proof}
This follows by the same arguments as in \cite[Lemma~2.11]{ACS20}.
\end{proof}

Now, define
\begin{equation}
\label{eq:M_n m_n}
M_n \; := \; \sup_{(t,x)\in Q(n)} u(t,x)\quad\text{and}\quad m_n \;:= \;\inf_{(t,x)\in Q(n)}u(t,x).
\end{equation}
For the purposes of the next lemma, given $k_0\in\R$, we denote 
\begin{equation}
\label{eq:k_js}
k_j\;:= \; M_n-2^{-j}(M_n-k_0), \qquad j\in\N.
\end{equation}
Also recall the definition of $\cA_3^\om(n)$ right before Theorem~\ref{thm:oscillations}.

\begin{lemma}
\label{lemma:u bound space}
 Let $t_0\in\R,x_0\in \bZ^d,$ and $u$ be such that $\partial_t u-\mathcal{L}_\theta^\omega u=0$ on $Q(n)$ for $n\geq 4$. Let $\eta:\bZ^d\rightarrow[0,1]$ be the spatial cut-off function $\eta(x):=[1- 2 d(x_0,x)/n]_+$.  
Suppose, for some $k_0\in \R$,
\begin{equation}
\label{eq:ass u1}
\frac{1}{n^2}\int_{t_0-n^2}^{t_0}\norm{\indicator_{\{u_t\leq k_0\}}}_{1,B(n),\eta^2\theta} \, dt \; \geq \; \frac{1}{2}.
\end{equation}
Then there exist $c_6,c_7 >0$ such that for any $\delta\in (0,1/4 c_7 \cA_3^\om(n))$ and any 
\begin{align*}
j\geq 1+  \frac{c_6 \, \norm{1\vee \mu^\omega}_{1,B(n)} \, \norm{1\vee (1/\theta^\om)}_{1,B(n)}}{\frac 1 4 - c_7 \delta \cA_3^\om(n)}
\end{align*}
we have that 
\begin{align*}
\norm{\indicator_{\{u_t\leq k_j\}}}_{1,B(n/2),\theta} \; \geq \; \delta,\qquad\forall t\in\big[t_0-\tfrac{1}{4}n^2,t_0\big].
\end{align*}
\end{lemma}

\begin{proof}
Set
$$v_t(x):=\frac{M_n-u_t(x)}{M_n-k_0},\qquad h_j=\epsilon_j:=2^{-j},\quad j\in\N.$$
Then $\partial_t(v+\eps_j)-\mathcal{L}_\theta^\omega(v+\eps_j)=0$  on $Q(n)$ for all $j\in\N$ and, for any $x\in\bZ^d$,  $u_t(x)>k_j$ if and only if $v_t(x)<h_j$. By \eqref{eq:ass u1} there exists $s_0\in[t_0-n^2,t_0-\frac{1}{3}n^2]$ such that
\begin{align} \label{eq:est_vs0}
\norm{\indicator_{\{v_{s_0}<1\}}}_{1,B(n),\eta^2\theta} \; \leq \; \frac{3}{4}.
\end{align}
To see this, assume the contrary is true, that is $\norm{\indicator_{\{v_s<1\}}}_{1,B(n),\eta^2\theta}>\frac{3}{4}$ for all $s\in[t_0-n^2,t_0-\frac{1}{3}n^2]$. Then
\begin{align*}
&\frac{1}{2} \;\geq \; \frac{1}{n^2}\int_{t_0-n^2}^{t_0}\norm{\indicator_{\{u_t> k_0\}}}_{1,B,\eta^2\theta} \, dt \; =\; \frac{1}{n^2}\int_{t_0-n^2}^{t_0}\norm{\indicator_{\{v_t<1\}}}_{1,B,\eta^2\theta} \, dt \\
& \mspace{36mu}
\; > \; \frac{1}{n^2}\int_{t_0-n^2}^{t_0-\frac{1}{3}n^2}\frac{3}{4} \, dt \; =\; \frac{1}{2},
\end{align*}
which is a contradiction.
Let $t\in[t_0-\frac{1}{4}n^2,t_0]$. By integrating the estimate \eqref{eq:energy est} over the interval $[s_0,t]$, noting that $\norm{\nabla\eta}_{l^\infty(E)}\leq2/n$, $\osr(\eta)\leq2$ and $t-s_0\leq n^2$, 
\begin{align*}
\norm{g(v_t+\eps_j)}_{1,B(n),\eta^2\theta}& \; \leq \; \norm{g(v_{s_0}+\eps_j)}_{1,B(n),\eta^2\theta} \,+ \, c \, \norm{1\vee \mu^\omega}_{1,B(n)} \, \norm{\theta^\om}_{1,B(n)}^{-1}.
\end{align*}
 Since $g$ is non-increasing and identically zero on $[1,\infty)$, using \eqref{eq:est_vs0} we have
\begin{align*}
\norm{g(v_{s_0}+\eps_j)}_{1,B(n),\eta^2\theta} \; \leq \; g(\eps_j)\, \norm{\indicator_{\{v_{s_0}<1\}}}_{1,B(n),\eta^2\theta} \; \leq \;  \frac{3}{4}\, g(\eps_j),
\end{align*}
and
\begin{align*}
\norm{g(v_{t}+\eps_j)}_{1,B(n),\eta^2\theta} \; \geq \; g(h_j+\eps_j)\, \norm{\indicator_{\{v_t<h_j\}}}_{1,B(n),\eta^2\theta}.
\end{align*}
So, combining the above, for $j\geq 2$
\begin{align*}
\norm{\indicator_{\{v_t<h_j\}}}_{1,B(n),\eta^2\theta}& \; \leq \; 
\frac 3 4 \, \frac{g(\eps_j)}{g(h_j+\eps_j)} \, + \, \frac{c}{g(h_j+\eps_j)} \,  \norm{1\vee \mu^\omega}_{1,B(n)} \, \norm{\theta^\om}_{1,B(n)}^{-1} \\
&\; \leq \; \frac{3}{4} \, \Big(1+\frac{1}{j-1}\Big) \, + \, \frac{c}{j-1} \, \norm{1\vee \mu^\omega}_{1,B(n)} \, \norm{1\vee (1/\theta^\om)}_{1,B(n)}.
\end{align*}
Then, since $\eta\equiv0$ on $B(n/2)^c$, 
\begin{align} \label{eq:lemma1}
& \norm{\indicator_{\{u_t\leq k_j\}}}_{1,B(n/2),\theta} \; = \; \frac{\langle\eta^2\theta^\om,1\rangle_{\ell^2(\bbZ^d)}}{\theta^\om(B(n/2))} \, \Big(1-\norm{\indicator_{\{v_t< h_j\}}}_{1,B(n),\eta^2\theta}\Big) \nonumber \\
& \mspace{36mu} \; \geq \;\frac{\langle\eta^2\theta^\om,1\rangle_{\ell^2(\bbZ^d)}}{\theta^\om(B(n/2))} \,  \bigg(\frac{1}{4} -\, \frac{c}{j-1} \, \norm{1\vee \mu^\omega}_{1,B(n)} \, \norm{1\vee (1/\theta^\om)}_{1,B(n)} \bigg).
\end{align}
Note that $\langle\eta^2\theta^\om,1\rangle_{\ell^2(\bbZ^d)}/\theta^\om(B(n/2)) \in(0,1)$ and  since $\eta\geq 1/2$ on $B(n/4)$,
\begin{align*}
\frac{\langle\eta^2\theta^\om,1\rangle_{\ell^2(\bbZ^d)}}{\theta^\om(B(n/2))} \; \geq \; c \, \norm{\theta^\om}_{1,B(n/4)} \, \norm{\theta^\om}_{1,B(n/2)}^{-1}.
\end{align*}
%
%
By combining this inequality above with (\ref{eq:lemma1}) and using that 
\begin{align*}
j-1 \; \geq \;   \frac{c_6 \, \norm{1\vee \mu^\omega}_{1,B(n)} \, \norm{1\vee (1/\theta^\om)}_{1,B(n)}}{\frac 1 4 - c_7 \delta \norm{\theta^\om}_{1,B(n/4)}^{-1} \, \norm{\theta^\om}_{1,B(n/2)} }
\end{align*}
 by Jensen's inequality, we get the claim.
\end{proof}

\begin{lemma}
\label{lemma:u bound cylinder}
Set $\tau:=1/4$ and $\sigma:=1/2$. Let $t_0\in \R,\,x_0\in \bZ^d,\,n\geq 4$, and suppose $u$ satisfies $\partial_t u-\mathcal{L}_\theta^\omega u=0$ on $Q(n)$. Assume there exist $\delta>0$ and $i_0\in\N$  such that 
\begin{align} \label{eq:ass:apriori}
\norm{\indicator_{\{u_t\leq k_{i_0}\}}}_{1,B(x_0,\sigma n),\theta} \;\geq \; \delta,\qquad\forall t\in I_\tau=\big[t_0-\tfrac{1}{4}n^2,t_0\big].
\end{align}
Let $\eps\in(0,1)$ be arbitrary. Then there exists 
\begin{align*}
j_0=j_0\big(\eps,\delta, i_0, \cA^\om_1(n), \norm{\mu^\omega}_{1,B(n)}, \norm{\theta^\om}_{1,B(n)} \big)\in\bbN  \quad  \text{with $j_0\geq i_0$,}
\end{align*} 
which is continuous and decreasing in the first two components and continuous and increasing in the other components,
such that
\begin{align*}
\norm{\indicator_{\{u>k_j\}}}_{1,1,Q_{\tau,\sigma}(n),\theta} \; \leq \; \eps, \qquad\forall\,  j\geq j_0.
\end{align*}
\end{lemma}
\begin{proof}
Let $\eta:\bZ^d\rightarrow[0,1]$ be a cut-off function such that $\supp\eta\subseteq B(n)$, $\eta\equiv1$ on $B_{\sigma}$ and $\eta\equiv 0$ on $\partial B(n)$  with linear decay on $B(n)\setminus B_{\sigma}$. So $\norm{\nabla \eta}_{\ell^\infty(E_d)}\leq 2/n$ and $\osr(\eta)\leq 2$. Now, let
\begin{align*}
w_t(x)\; := \; \frac{M_n-u_t(x)}{M_n-k_{i_0}}\quad \text{and} \quad h_j=\eps_j:=2^{-j}.
\end{align*}
Then $w\geq 0$ and $\partial_t(w+\eps_j)-\mathcal{L}_\theta^\omega(w+\eps_j)=0$ on $Q(n)$ for $j\in\N$. For any $t\in I_\tau$, let $\mathcal{N}_t:=\{x\in B_\sigma:g(w_t(x)+\eps_j)=0\}$. Since  $g\equiv 0$ on $(1,\infty)$ by its definition,  
\begin{align*}
\frac{\theta^\om(\mathcal{N}_t)}{\theta^\om(B_\sigma)} 
\;= \;
\norm{\indicator_{\{g(w_t+\eps_j)=0\}}}_{1,B_\sigma,\theta} 
\; \geq \; 
\norm{\indicator_{\{w_t\geq 1\}}}_{1,B_\sigma,\theta}
\; = \;
\norm{\indicator_{\{u_t\leq k_{i_0}\}}}_{1,B_\sigma,\theta}
\;\geq \; \delta,
\end{align*}
where we used \eqref{eq:ass:apriori}  in the last step. By Proposition \ref{prop:poincare} we have
\begin{align*}
\norm{g(w_t+\eps_j)}_{1,B_\sigma,\theta}^2 \;
\leq \;  c_7 \, n^2 \,  \cA^\om_1(\sigma n) \, \bigg(1+\frac{\theta^\om(B_\sigma)}{\theta^\om(\mathcal{N}_t)}\bigg)^{\!2}  \,
  \frac{\mathcal{E}^{\omega,\eta^2}\big(g(w_t+e_j)\big)}{\abs{B_\sigma}},
\end{align*}
so that by Jensen's inequality and by integrating (\ref{eq:energy est}) over $I_\tau$, 
\begin{align*}
& \norm{g(w+\eps_j)}_{1,1,Q_{\tau,\sigma},\theta}^2
\;\leq \;
\frac{1}{\tau n^2}\int_{I_\tau}\norm{g(w_t+\eps_j)}_{1,B_\sigma,\theta}^2 \, dt \\
&\mspace{36mu}
\; \leq \; \frac{c}{\delta^2} \, \cA_1^\om(\sigma n) \, \norm{\theta^\om}_{1,B(n)} \, \int_{I_\tau}\frac{\mathcal{E}^{\omega,\eta^2}\big(g(w_t+\eps_j)\big)}{\theta^\om(B(n))} \, dt \\
&\mspace{36mu}
\; \leq \; \frac{c}{\delta^2} \, \cA_1^\om(n) \, \Big( \norm{\theta^\om}_{1,B(n)} \, \norm{\eta^2g(w_{t_0-\tau n^2}+\eps_j)}_{1,B(n),\theta}\,+ \,
\norm{1\vee\mu^\omega}_{1,B(n)}\Big).
\end{align*}
Since $g$ is non-increasing and $w_t>0$ for all $t\in I_\tau$,
\begin{align} \label{eq:lemma2eq}
& \norm{\mathbbm{1}_{w<h_j}}_{1,1,Q_{\tau,\sigma},\theta}^2
\; \leq \;
\frac{\norm{g(w+\eps_j)}_{1,1,Q_{\tau,\sigma},\theta}^2}{g(h_j+\eps_j)^2} \nonumber \\
&
\mspace{36mu} \; \leq \; \frac{c}{\delta^2} \, A_1^\omega(n) \, \bigg(\norm{\theta}_{1,B(n)} \, \frac{g(\eps_j)}{g(h_j+\eps_j)^2}
\,+ \,
\norm{1\vee\mu^\omega}_{1,B(n)} \,\frac{1}{g(h_j+\eps_j)^2}\bigg) \nonumber \\
& \mspace{36mu}
\;\leq \; \frac{c}{\delta^2} \, A_1^\omega(n) \, \norm{ 1 \vee \theta^\om}_{1,B(n)} \, \norm{1\vee\mu^\omega}_{1,B(n)} \,  \bigg(\frac{j}{(j-1)^2} \, + \, \frac{1}{(j-1)^2}\bigg).
\end{align}
Thus, for any $\eps>0$, there exists some  $j_0 \geq i_0$ as in the statement such that  $\norm{\indicator_{\{u>k_j\}}}_{1,1,Q_{\tau,\sigma},\theta}=\norm{\indicator_{\{w<h_{j-{i_0}}\}}}_{1,1,Q_{\tau,\sigma},\theta}\leq \eps$ for all $j\geq j_0$. 
\end{proof}

\begin{proof}[Proof of Theorem \ref{thm:oscillations}]
We may assume without loss of generality that $u>0$, otherwise consider $u-\inf_{Q(n)}u$. Set $\tau=1/4$, $\sigma=1/2$ as before in Lemma~\ref{lemma:u bound cylinder}.  Define $k_0:=(M_n+m_n)/2$ with $M_n$ and $m_n$ as in \eqref{eq:M_n m_n} and let $k_j$ be defined by \eqref{eq:k_js}. Further, let $\eta$ be the cut-off function $\eta(x):=\left[1-d(x_0,x)/\sigma n\right]_+$.
We may assume 
$$\frac{1}{n^2}\int_{I_1}\norm{\indicator_{\{u_t\leq k_0\}}}_{1,B(n),\eta^2\theta}\, dt \; \geq \; \frac{1}{2}.$$
Otherwise, consider $M_n+m_n-u$ in place of $u$. 
%
Set $\eps:=\big(2 c_3 \, (4 \cA_2^\om(\sigma n))^\kappa \big)^{-2p_*}$ with  $\cA_2^\om(n)$ as in Theorem \ref{thm:maximal}. Fix any $\Delta\in(0,\frac{2}{d+2})$ and $N_3\geq 2N_1(\Delta)$ such that $\frac{1}{2}>(\sigma N_3)^{-\Delta}$. Now for all $n\geq N_3$, applying consecutively Lemma~\ref{lemma:u bound space} and Lemma~\ref{lemma:u bound cylinder}, there exists 
\begin{align*} 
l=l^\om(x_0,n)=l\big(\cA^\om_1(n), \cA_2^\om(n), \cA^\om_3(n), \norm{\mu^\omega}_{1,B(n)}, \norm{\theta^\om}_{1,B(n)}, \norm{1/\theta^\om}_{1,B(n)}\big),
\end{align*} 
which is continuous and increasing in all components,
such that
$$\norm{\indicator_{\{u>k_j\}}}_{1,1,Q_{\tau,\sigma}(n),\theta} \;\leq \; \eps,\qquad\forall j\geq l.$$
By an application of Jensen's inequality,
\begin{align*}
& \norm{(u-k_l)_+}_{2p_*,2,Q_1(\sigma n),\theta} \; \leq  \; \big(M_n-k_l\big) \, \norm{\indicator_{\{u>k_l\}}}_{2p_*,2,Q(\sigma n),\theta} \\
& \mspace{36mu} \;\leq  \; \big(M_n-k_l\big) \, \norm{\indicator_{\{u>k_l\}}}_{1,1,Q_1(\sigma n),\theta}^{1/2p_*} \; \leq \; \big(M_n-k_l\big) \, \eps^{1/2p_*}.
\end{align*}
Now, let $\vartheta =\frac{\sigma}{2}=\frac{1}{4}$. Then Theorem~\ref{thm:maximal} implies that
\begin{align*}
& M_{\vartheta n} \; \leq \; \max_{Q_{1/2}(\sigma n)}u(t,x)
\; \leq \; k_l+ c_3  \big(4 \cA_2^\om(\sigma n) \big)^\kappa \, \norm{(u-k_l)_+}_{2p_*,2,Q_1(\sigma n),\theta}\\
& \mspace{36mu} \;\leq \; k_l \,+ \,\frac{1}{2} \, \big(M_n-k_l\big) \;= \; M_n\,- \,2^{-(l+2)}\, \big(M_n-m_n\big).
\end{align*}
Hence
\begin{align*}
M_{\vartheta n}-m_{\vartheta n}\;\leq \; M_n \,- \, 2^{-(l+2)} \, \big(M_n-m_n\big)-m_{\vartheta n} \;\leq \; \big(1-2^{-(l+2)}\big) \, \big(M_n-m_n\big),
\end{align*}
and the theorem is proven.
\end{proof}

\subsection{Proof of the Local Limit Theorem}
\label{subsection:qllt}
As mentioned at the beginning of this section, we will derive the required H\"older regularity estimate from the oscillation inequality in Theorem~\ref{thm:oscillations}. The following version of the ergodic theorem will help us to control ergodic averages on scaled balls with varying centre-points.

\begin{proposition} \label{prop:krengel_pyke}
  Let $\cB:=  \big\{ B :  \text{$B$ closed Euclidean ball in $\bbR^d\!$}\big\}$.
  Suppose that Assumption~\ref{ass:ergodicity general speed measure} holds.  Then, for any $f \in L^1(\Om)$,
  \begin{align*}
    \lim_{n \to \infty} \sup_{ B \in \cB}
    \bigg|
      \frac{1}{n^{d}}\, \sum_{x \in (nB) \cap \bbZ^d}\mspace{-12mu}
        f \circ \tau_{x}
      \,-\,
      |B| \cdot \mean\!\big[f\big]
    \bigg|
    \;=\;
    0, \qquad \prob\text{-a.s.},
  \end{align*}
  where $|B|$ denotes the Lebesgue measure of $B$.
\end{proposition}
\begin{proof}
 See, for instance, \cite[Theorem~1]{KP87}.
\end{proof}

\begin{lemma}
\label{lem:ergodic}
Suppose Assumptions~\ref{ass:ergodicity general speed measure} and \ref{ass:limsups} hold. Let $\gamma^\om$ be as in Theorem~\ref{thm:oscillations}. 
Then, $\prob$-a.s., for any  $x\in\R^d$ and $\delta\in (0,1)$,
\begin{align*}
\limsup_{n\to\infty} \gamma^\om(\floor{nx}, \delta n) \; \leq \; \bar \gamma \in (0,1),
\end{align*}
with $\bar \gamma$ only depending on the law of $\om$ and $\theta^\om$. 
\end{lemma}
\begin{proof}
Recall that $\ga^\om$ is continuous and increasing in all components.  By Proposition~\ref{prop:krengel_pyke} we have, for instance, for any $x\in \bbR^d$ and $\de\in (0,1)$,
\begin{align*}
\limsup_{n\to\infty} \norm{\mu^\om}_{1,  B(\lfloor n x  \rfloor, \de n)}\; \leq  \; \E\big[\mu^\om(0)\big]=:\bar \mu, \qquad \text{ $\bbP$-a.s.}
\end{align*}
Analogous statements hold for the other components of $\ga^\om$, that is $\cA_1^\om$, $\cA_2^\om$ etc. Since $\ga^\om$ is continuous and increasing in all components we get the claim for some $\bar \gamma \in (0,1)$ depending only on the respective moments of $\mu^\om(0)$, $\nu^\om(0)$ and $\theta^\om(0)$ .
\end{proof}

Lemma~\ref{lem:ergodic} facilitates applying the oscillations inequality iteratively with a common, deterministic constant. Together with the upper heat kernel bound cited below, this will produce a H\"older continuity statement for the rescaled heat kernel in Proposition~\ref{prop:cont_hk} below.

\begin{lemma}\label{lem:ondiag_quenched} Suppose Assumptions \ref{ass:ergodicity general speed measure} and \ref{ass:limsups} hold. For $\prob$-a.e.\ $\om$, any $\lambda>0$ and $x\in \bbZ^d$ there exist $c_8=c_8(d,p,q,r, \lambda)$ and $N_5=N_5(x,\om)$ such that for any $t$ with $\sqrt{t}\geq N_5$  and all $y\in B(x,\lambda \sqrt{t})$,
\begin{align*}
p_\theta^\om(t,x,y) \; \leq \; c_8 \, t^{-d/2}.
\end{align*}
\end{lemma}
\begin{proof}
This can be directly read off \cite[Theorem~3.2]{ADS19} or derived from Theorem \ref{thm:maximal} by the method in \cite[Corollary 2.10]{ACS20}.
\end{proof}

\begin{proposition} \label{prop:cont_hk}
Let $\delta>0$, $\sqrt{t}/2 \geq \delta$ and $x\in \bbR^d$ be fixed. Then, there exists $c_9>0$ such that for $\prob$-a.e.\ $\om$,
  \begin{align*}
\limsup_{n\to \infty} \sup_{\substack{y_1,y_2 \in B(\floor{nx},\delta n) \\ s_1,s_2 \in [t-\delta^2, t]}} n^d \, \big| p_\theta^\omega\big(n^2s_1,0,y_1 \big)-p_\theta^\omega\big(n^2 s_2,0,y_2\big)\big|
    \;\leq\;
    c_9\, \bigg( \frac{\delta}{\sqrt{t}} \bigg)^{\!\!\varrho}\,
   t^{-d/2},
  \end{align*}
  where $\varrho= \ln( \bar\gamma)/\ln(1/4)$.
\end{proposition}
\begin{proof}
Set $\delta_k:=4^{-k}\sqrt{t}/2$ and with a slight abuse of notation let
\begin{align*}
Q_k \; := \; n^2 [t- \delta_k^2, t] \times B(\floor{nx},\delta_k n), \qquad k\geq 0.
\end{align*}
Choose $k_0\in\N$ such that $\delta_{k_0}\geq \delta > \delta_{k_0+1}$.
In particular, for every $k\leq k_0$ we have $\delta_{k} \in [\delta,\sqrt{t}]$. Now apply Theorem~\ref{thm:oscillations} and Lemma~\ref{lem:ergodic}, which give that there exists $N_6=N_6(\om, x,\delta)$ such that  for $\prob$-a.e.\ $\om$ and all $n \geq N_6$,
\begin{align*}
\osc_{Q_{k}} p_\theta^\omega\big(\cdot,0,\cdot \big) \; \leq \;  \bar \gamma \, \osc_{Q_{k-1}} p_\theta^\omega\big(\cdot,0,\cdot \big), \qquad \forall k=1,...,k_0.
\end{align*}
 We  iterate the above inequality on the chain  $Q_0 \supset Q_1 \supset \cdots \supset Q_{k_0}$ to obtain
  \begin{align} \label{eq:pre_hoelder}
  \osc_{Q_{k_0}} p_\theta^\omega\big(\cdot,0,\cdot \big)
    \;\leq\;
    \bar \gamma^{k_0}\, \max_{Q_{0}} p_\theta^\omega\big(\cdot,0,\cdot \big).
  \end{align}
  Note that 
\begin{align*}
Q_{k_0} \; = \; n^2 [t- \delta_{k_0}^2, t] \times B(\floor{nx},\delta_{k_0} n) \; \supset \; n^2 [t-\delta^2,t] \times B(\floor{nx},\delta n).
\end{align*}  
Hence, since $\overline{\gamma}^{k_0} \leq c (\delta / \sqrt{t})^\varrho$, the claim follows from \eqref{eq:pre_hoelder} and Lemma~\ref{lem:ondiag_quenched}.
\end{proof}

We shall now apply the above H\"older regularity to prove a pointwise version of the local limit theorem.
\begin{proposition}
\label{prop:pointwise llt theta}
Suppose Assumptions \ref{ass:ergodicity general speed measure} and \ref{ass:limsups} hold. For any $x\in\Z^d$ and $t>0$,
\begin{align*}
\lim_{n\to\infty} \, \abs{n^d \, p_\theta^\omega\big(n^2t,0,\floor{nx}\big)-a \, k_t(x)} \;= \; 0,\qquad\text{$\prob$-a.s.}
\end{align*}
with $k_t$ as defined in \eqref{eq:def_kt} and $a:=\E\big[\theta^\om(0)\big]^{-1}$.
\end{proposition}

\begin{proof}
For any $x\in\R^d$ and $\delta>0$ let $C(x,\delta):=x+\left[-\delta,\delta\right]^d$ and $C^n(x,\delta):=n \, C(x,\delta)\cap\Z^d$, i.e.\ $C(x, \de)$ is a ball in $\bbR^d$ with respect to the supremum norm. Note that the cubes $C^n(x,\delta)$ are comparable with $B(\floor{nx},\delta n)$ and we may apply Proposition~\ref{prop:cont_hk} with $B(\floor{nx},\delta n)$ replaced by $C^n(x,\delta)$.  Let 
$$J :=  \Big(p_\theta^\omega\big(n^2t,0,\floor{nx}\big)- n^{-d} \,a \, k_t(x)\Big) \, \theta^\om\big(C^n(x,\delta)\big).$$
We can rewrite this, for any $\delta>0$, as $J=\sum_{i=1}^4 J_i$ where
\begin{align*}
J_1&:=\sum_{z\in C^n(x,\delta)}\Big(p_\theta^\omega\big(n^2t,0,\floor{nx}\big)-p_\theta^\omega\big(n^2t,0,z\big)\Big) \, \theta^\om(z),\\
J_2&:=P_0^\omega\Big[X_t^{(n)}\in C(x,\delta)\Big]-\int_{C(x,\delta)}k_t(y) \, dy,\\
J_3&:=k_t(x)\, \Big((2\delta)^d-\theta^\om\big(C^n(x,\delta)\big) \, n^{-d} \, a\Big),\\
J_4&:=\int_{C(x,\delta)}\big(k_t(y)-k_t(x)\big) \, dy,
\end{align*}
with $X_t^{(n)}:=\frac{1}{n}X_{n^2t}$, $t\geq 0$, being the rescaled random walk. 
It suffices to prove that, for each $i=1,...,4$,  as $n\to \infty$, $|J_i|/n^{-d}\theta^\om\big(C^n(x,\delta)\big)$ converges $\mathbb{P}$-a.s.\ to a  limit which is small with respect to $\delta$.

First note that $J_2\rightarrow 0$ by Theorem~\ref{thm:qfclt general speed} and $ n^{-d}\theta^\om\big(C^n(x,\delta)\big)\to(2\delta)^d/a$ by the arguments of Lemma~\ref{lem:ergodic}. Thus, $\lim_{n\to \infty}|J_i|/n^{-d}\theta^\om\big(C^n(x,\delta)\big) = 0$ for $i=2,3$. Further, by the Lipschitz continuity of the heat kernel $k_t$ in its space variable it follows that
 $\lim_{n\to\infty} |J_4|/ n^{-d}\, \theta^\om\big(C^n(x,\delta)\big)=O(\delta)$.
To deal with the remaining term, we  apply Proposition~\ref{prop:cont_hk}, which yields
\begin{align*}
 \limsup_{n\to \infty} \max_{z \in C^{n}(x, \de)}\, n^{d}\,
  \big|
    p_\theta^{\om}\big( n^2 t, 0, z \big)
    \,-\,
    p_\theta^{\om}\big( n^2 t, 0, \lfloor n x \rfloor \big)
  \big|
  \;\leq\;
  c\, \de^{\varrho}\, t^{-\frac{d}{2} - \frac{\varrho}{2}}.
\end{align*}
Hence, $\limsup_{n\to\infty} |J_1|/ n^{-d}\, \theta^\om\big(C^n(x,\delta)\big)=O(\delta^{\varrho})$, $\prob$-a.s.
Finally, the claim follows by letting $\delta \to 0$.
\end{proof}

\begin{proof}[Proof of Theorem \ref{thm:qllt general speed}]
Having proven the pointwise result Proposition~\ref{prop:pointwise llt theta}, the full local limit theorem follows by extending over compact sets in $x$ and $t$. This is done using a covering argument, exactly as in Step~2 in the proof of \cite[Theorem~3.1]{ACS20}, which in turn is a slight modification of the proofs in \cite{croydon2008local} and \cite{barlow2009parabolic}.
\end{proof}

\subsection{Weak parabolic Harnack inequality and near diagonal heat kernel bounds} \label{sec:corollaries of qllt} 
The above method of proving the local limit theorem is simpler than the derivations of \cite{harnack, barlow2009parabolic}, in part because it does not require a full parabolic Harnack inequality. However, the above analysis still provides a weak parabolic Harnack inequality.

\begin{proposition}
\label{prop:weak harnack}
Suppose Assumptions \ref{ass:ergodicity general speed measure} and \ref{ass:limsups} hold. For any $x_0\in \bbZ^d$, $t_0 \in \bbR$ and $\bbP$-a.e.\ $\om$, there exists $N_7=N_7(\om,x_0)$ such that for all $n \geq N_7$ the following holds. Let $u>0$ be such that $\partial_t u-\cL_\th^\om u=0$ on $Q(n):=[t_0-n^2,t_0]\times B(x_0,n)$. Assume there exists $\eps >  0$ such that 
\begin{equation}
\frac{1}{n^2}\int_{t_0-n^2}^{t_0}\norm{\indicator_{\{u_t\geq \eps\}}}_{1,B(n),\eta^2\theta} \, dt \; \geq \; \frac{1}{2}
\end{equation}
with $\eta$ as in Lemma~\ref{lemma:u bound space}. Then there exists $\gamma=\gamma(\eps,p,q,d)$ (also depending on the law of $\om$ and $\theta^\om$) such that 
$$u(t,x)\geq \gamma\qquad \forall\,(t,x)\in Q_{\frac{1}{2}}(n/2)=[t_0-n^2/8,t_0]\times B(x_0,n/4).$$
\end{proposition}
\begin{proof}
This follows by the same method as \cite[Theorem 2.14]{ACS20}.  Theorem \ref{thm:maximal}, Lemma \ref{lemma:u bound space} and Lemma \ref{lemma:u bound cylinder} are all necessary ingredients.
\end{proof}

Finally, we can also derive from Theorem~\ref{thm:qllt general speed} a near-diagonal lower heat kernel estimate, which complements the upper bounds obtained in \cite{ADS19}.

\begin{corollary} \label{cor:near diag ests}
Suppose Assumptions \ref{ass:ergodicity general speed measure} and \ref{ass:limsups} hold. For $\bbP$-a.e.\ $\om$, there exists $N_8(\om)>0$  and $c_{10}=c_{10}(d)>0$ such that  for all $t\geq N_8(\om)$ and $x\in B(0,\sqrt{t})$,
\begin{align*}
p_\th^\om(t,0,x)\geq c_{10}\, t^{-d/2}.
\end{align*}
\end{corollary}
\begin{proof}
This follows from the local limit theorem exactly as for the constant speed case in \cite[Lemma 5.3]{harnack}.
\end{proof}

\section{Annealed Local Limit Theorem under General Speed Measure}
\label{section:annealed general speed}

\subsection{Maximal Inequality for the Heat Kernel}
The first step to show the annealed local limit theorem in Theorem~\ref{thm:annealed llt general speed} is to establish an $L^1$ form of the maximal inequality in \cite{ADS19}, which involves space-time cylinders of a more convenient form for this section. So for $\eps\in (0,1/4 ),\; x_0\in\bbZ^d$, we redefine
$$Q_\sigma(n):=\big[(1-\sigma)\eps n^2,n^2-(1-\si)\eps n^2 \big]\times B(x_0,\sigma n)$$
where $n\in\bbN$ and $\sigma\in[\frac{1}{2},1]$.

\begin{proposition}
\label{prop:general speed maximal L1 2}
Fix $\eps\in(0,1/4)$, $x_0\in \bbZ^d$ and let $p,q,r \in (1,\infty]$ be such that \eqref{eq:pqr_eqn} holds. There exists $c_{11}=c_{11}(d,p,q,r)$  such that for all $n\geq 1$ and $1/2\leq \sigma' <\sigma \leq 1$,
\begin{align*}
\max_{(t,x)\in Q_{\sigma'}(n)}p_\theta^\om(t,0,x)  \leq  c_{11}   \norm{1\vee (1/\theta^\omega)}_{1,B(n)} \bigg( \frac{\cA_4^\om(n)}{\eps(\sigma-\sigma')^2}\bigg)^{\!\kappa'}  \onorm{p_\theta^\om(\cdot, 0, \cdot)}_{1,1/p_*,Q_\sigma(n),\theta},
\end{align*}
where $\kappa'=\kappa'(d,p,q,r):=p_*+p_*^2\rho/(\rho-r_*p_*)$  with $\rho$ as in Proposition~\ref{prop:sobolev} and
\begin{align} \label{eq:defA4}
\cA_4^\om(n) \; :=  \;\onorm{1\vee (\mu^\omega/\theta^\om)}_{p,B(n),\theta} \, \norm{1\vee \nu^\omega}_{q,B(n)} \,   \norm{1 \vee \theta^\om}_{r,B(n)} \norm{1\vee (1/\theta^\om)}_{q,B(n)}.
\end{align}
\end{proposition}
\begin{proof}
For abbreviation we set $u=p_\theta^\om(\cdot, 0,\cdot)$ and $\sigma_k:=\sigma-(\sigma-\sigma^\prime)2^{-k}$. Further, write $B_k:=B(x_0,\sigma_k n)$ and $Q_k:=Q_{\sigma_k}(n)$. Note that $\abs{B_k}/\abs{B_{k+1}}\leq c\,2^d$. Let $\gamma=1/(2p_*)$. Then by H\"older's inequality
\begin{align*}
\onorm{u}_{2p_*,2,Q_{k},\theta} &\; \leq \; \onorm{u}_{1,2\gamma,Q_{k},\theta}^\gamma \, \norm{u}_{\infty,\infty,Q_{k}}^{1-\gamma},
\end{align*}
and by the proof of \cite[Proposition 3.8]{ADS19} (cf.\ last line on page~14), setting $\phi=1$ and $\delta=1$ there, we have
\begin{align*}
\norm{u}_{\infty,\infty,Q_{k-1}} \; \leq \; c \, \bigg(\frac{\cA_4^\omega(n)}{\eps(\sigma_k-\sigma_{k-1})^2}\bigg)^{\!\kappa/p_*}\onorm{u}_{2p_*,2,Q_{k},\theta}
\end{align*}
with $\kappa=\kappa(d,p,q,r)$ as throughout \cite{ADS19}. Combining the above equations yields
\begin{align*}
\norm{u}_{\infty,\infty,Q_{k-1}} \; \leq \; 2^{2\kappa k/p_*} \, J \,\onorm{u}_{1,2\gamma,Q_{\sigma},\theta}^\gamma \, \norm{u}_{\infty,\infty,Q_{k}}^{1-\gamma},
\end{align*}
where we have introduced $J:= c \big(\mathcal{A}_4^\om(n)/\eps(\sigma-\sigma')^2 \big)^{\kappa/p_*}\geq 1$ for brevity.
By iteration, we have for any $K\in\bbZ_+$,
\begin{align} \label{eq:post_it}
\norm{u}_{\infty,\infty,Q_{\sigma^\prime}} \; \leq  \; 2^{2\kappa/p_*\sum_{k=0}^{K-1}(k+1)(1-\gamma)^k} \Big(J \, \onorm{u}_{1,2\gamma, Q_\sigma,\theta}^\gamma\Big)^{\sum_{k=0}^{K-1}(1-\gamma)^k}\norm{u}_{\infty,\infty,Q_{K}}^{(1-\gamma)^K}.
\end{align}
Note that $p_\theta^\om (t,0,x) \, \theta^\om(x) \leq 1$ for all $t>0$ and $x\in \bbZ^d$. Therefore,
\begin{align*}
\norm{u}_{\infty,\infty,Q_{K}}\;&\leq\; \max_{x\in B_{K}} \theta^\omega(x)^{-1}  \,   \max_{(t,x)\in Q(n)} \, u(t,x) \, \theta^\omega(x) \; \leq \; \abs{B_{K}} \,  \norm{1/\theta^\omega}_{1,B_{K}}.
\end{align*}

Since $\limsup_{K\to\infty} \abs{B_{K}}^{(1-\gamma)^K}\leq c$ with $c$ independent of $n$ and $\norm{1/\theta^\omega}_{1,B_{K}}^{(1-\gamma)^K} \leq c \, \norm{1\vee (1/\theta^\omega)}_{1,B(n)}$, we obtain by letting $K\to\infty$ in \eqref{eq:post_it},
\begin{align*}
\norm{u}_{\infty,\infty,Q_{\sigma^\prime}} \; \leq \;  \,2^{\frac{2\kappa}{p_*\gamma^2}} \,  \norm{1\vee (1/\theta^\omega)}_{1,B(n)} \,
J^{1/\gamma} \, \onorm{u}_{1,2\gamma,Q_\sigma,\theta},
\end{align*}
which completes the proof, with $\kappa':=2\kappa.$
\end{proof}

\subsection{Proof of Theorem~\ref{thm:annealed llt general speed} }

Here we anneal the results of Section \ref{section:general speed measure} to derive  the annealed local limit theorem for the static RCM under a general speed measure stated in Theorem~\ref{thm:annealed llt general speed}. This will require a stronger moment condition. For any $p,q,r_1,r_2 \in [1,\infty]$ set
\begin{align*}
M(p,q,r_1,r_2) \; := \; \mathbb{E}\left[ \mu^\om(0)^{p}\right]+\mathbb{E}\left[\nu^\om(0)^{q}\right]
+\mathbb{E}\left[\theta^\om(0)^{-r_1}\right]+\mathbb{E}\left[\theta^\om(0)^{r_2}\right] \in (0,\infty]. 
\end{align*}


\begin{proposition}
\label{prop:general speed annealed maximal bound}
Suppose Assumption \ref{ass:ergodicity general speed measure} holds. Then there exist $p, q, r_1, r_2 \in (0,\infty)$ (only depending on $d$) such that, under the moment condition $M(p,q,r_1,r_2) < \infty$, for all $K>0$ and $0<T_1\leq T_2$,
\begin{align*}
\mathbb{E}\bigg[\sup_{n\geq 1,\,\abs{x}\leq K,\,t\in[T_1,T_2]}n^d \,p_\theta^\omega(n^2t,0,\floor{nx})\bigg] \; < \; \infty.
\end{align*}
\end{proposition}
Before we prove Proposition~\ref{prop:general speed annealed maximal bound}  we remark that it immediately implies the annealed local limit theorem.

\begin{proof}[Proof of Theorem~\ref{thm:annealed llt general speed}] 
Given the quenched  result in Theorem \ref{thm:qllt general speed}, the statement follows from Proposition~\ref{prop:general speed annealed maximal bound} by  the dominated convergence theorem.
\end{proof}


The rest of this section is devoted to the proof of Proposition \ref{prop:general speed annealed maximal bound}. We start with a  consequence of the maximal inequality in  Proposition~\ref{prop:general speed maximal L1 2}.

\begin{lemma} \label{lem:hkL1}
Let $p,\,q,\,r\in(1,\infty]$ be such that \eqref{eq:pqr_eqn} holds.
For all $K>0,\,0<T_1\leq T_2$, there exist $c_{12}=c_{12}(d,p,q,r,K,T_1,T_2)$ and $c_{13}=c_{13}(K,T_2)$  such that 
\begin{align*}
\sup_{\abs{x}\leq K,\,t\in\left[T_1,T_2\right]}n^d \, p_\theta^\omega(n^2t,0,\floor{nx})
\; \leq \; c_{12} \,  \norm{1\vee (1/\theta^\omega)}_{1,B(n)} \, \cA_4^\omega(c_{13} n)^{\kappa'}, \qquad \forall n\geq 1,
\end{align*}
 with $\cA_4^\omega$ as in \eqref{eq:defA4}.
\end{lemma}
\begin{proof}

First note that by definition of the heat kernel $p_\theta^\omega$,
\begin{align} \label{eq:hk_1norm}
& \onorm{ p_\theta^\omega(\cdot,0,\cdot)}_{1,1/p_*,Q_1(n),\theta} \; = \;\Bigg(\frac{1}{\abs{I_1}}\int_{I_1}\bigg(\frac{1}{\abs{B(n)}}\sum_{y\in B(n)}p_\theta^\omega(t,0,y) \, \theta^\omega(y) \bigg)^{\! 1/p_*} \, dt \Bigg)^{p_*} \nonumber\\
& 
\mspace{36mu}
\leq \; c \, n^{-d} \, \bigg(\frac{1}{\abs{I_1}}\int_{I_1} P_0^\omega\big[ X_t\in B(n)\big]^{1/p_*} \, dt \bigg)^{p_*} \;\leq \; c \, n^{-d},
\end{align}
for all $n\in\mathbb{N}$. 
Choose $x_0=0$ and set $N=c_{13}n$ for any $c_{13}> 2\ceil{K\vee\sqrt{T_2}}$. Then we can find $\eps\in(0,1/4)$ such that 
\begin{align*}
\big\{(n^2t,\floor{nx}):\, t\in [T_1,T_2],\,\abs{x}\leq K\big\}\; \subseteq \; Q_{1/2}(N)=\big[\tfrac{\eps}{2} N^2, (1-\tfrac{\eps}{2}) N^2\big]\times B(0,N/2).
\end{align*}
The claim follows now from Proposition~\ref{prop:general speed maximal L1 2} with the choice $\sigma=1$ and $\sigma'=1/2$ together with \eqref{eq:hk_1norm}. 
\end{proof}


\begin{proof}[Proof of Proposition~\ref{prop:general speed annealed maximal bound}]
By Lemma~\ref{lem:hkL1} it suffices to show that, under a suitable moment condition, $\mean\big[ \sup_{n\geq 1}  \norm{1\vee (1/\theta^\omega)}_{1,B(n)} \, \cA_4^\om(c_{13} n)^{\kappa'} \big]<\infty$.  Recall that
\begin{align*}
\cA_4^\om(n)\;= \; 
\onorm{1\vee (\mu^\omega/\theta^\om)}_{p,B(n),\theta}  \, \norm{1\vee \nu^\omega}_{q,B(n)}  \,   \norm{1 \vee \theta^\om}_{r,B(n)} \norm{1\vee (1/\theta^\om)}_{q,B(n)},
\end{align*} 
for any $p,q,r \in (1,\infty]$ satisfying \eqref{eq:pqr_eqn}. After an application of H\"older's inequality it suffices to show
that $\mean\big[ \sup_{n\geq 1} \, \norm{\nu^\omega}_{q,B(n)}^{4\kappa'} \big]  <  \infty$ and similar moment bounds on the other terms. Now suppose that $\mathbb{E}\big[ \nu^{\om}(0)^{4\kappa^\prime\vee q'}\big]<\infty$ for any $q'>q$. Then, if $4\kappa'>q$, given Assumption~\ref{ass:ergodicity general speed measure}, we can apply the $L^p$-version of the maximal ergodic theorem (see \cite[Chapter 6, Theorem~1.2]{krengel}, cf.\  Proposition~\ref{prop:dynamic ergodic} below) to deduce
\begin{align*}
\mean\Big[ \sup_{n\geq 1} \,  \norm{\nu^\omega}_{q,B(n)}^{4\kappa'} \Big]  \; \leq \;  c \, \mathbb{E}\big[ \nu^{\om}(0)^{4\kappa'}\big] \; < \; \infty.
\end{align*}
In the case  $4\kappa'\leq q<q'$, we  have by Jensen's inequality and again the maximal ergodic theorem,
\begin{align*}
\mean\Big[ \sup_{n\geq 1}  \norm{\nu^\omega}_{q,B(n)}^{4\kappa'} \Big]  
& \leq   \mathbb{E}\Bigg[ \sup_{n\geq 1} \bigg(\frac{1}{\abs{B(n)}}\sum_{x\in B(n)} \nu^\om(x)^{q}\bigg)^{\!\frac{q'}{q}}\Bigg]^{\frac{4\kappa'}{q'}} 
 \leq  c \, \mathbb{E}\big[v^\om(0)^{q'}\big]^{\frac{4\kappa'}{q'}}  < \infty.
\end{align*}
The other terms involving $\norm{\theta^\om}_{r,B(n)}$ etc.\ can be treated similarly.
\end{proof}

\section{Annealed Local Limit Theorem for the Dynamic RCM}
\label{section:dynamic}

Similarly as in the static case our starting point is establishing an $L^1$ maximal inequality for space-time harmonic functions. Once again, we redefine our space-time cylinders. For $t_0\in\bbR,\,x_0\in\bbZ^d,n\in\bbN,$ and $\sigma\in(0,1]$, let
$$Q_\sigma (n):=[t_0,t_0+\sigma n^2]\times B(x_0,\sigma n).$$
Throughout this section we fix $p,\,q\in(1,\infty]$ satisfying
\begin{equation}
\label{eq:pq_cond_dyn}
\frac{1}{p-1}\frac{q+1}{q}+\frac{1}{q}<\frac{2}{d}.
\end{equation}

\begin{proposition} \label{prop:maximal_dyn}
Let $t_0\in \R$, $x_0\in \bZ^d$ and $\Delta\in(0,1)$. There exist $N_9=N_9(\Delta)\in\bbN$ and $c_{14}=c_{14}(d,p,q)$  such that for all $n\geq N_9$ and $\frac 1 2\leq \sigma' <\sigma \leq 1$ with $\sigma-\sigma' > n^{-\Delta}$,
\begin{align*}
& \max_{(t,x)\in Q_{\sigma'}(n)}p^\om(0,t,0,x) \;\leq \;  \, c_{14} \, \bigg( \frac{\cA_5^\om(n)}{(\sigma-\sigma')^2} \bigg)^{\!\kappa'} \, \norm{p^\om(0,\cdot,0,\cdot)}_{1, 1,Q_\sigma(n)}^{\beta_n},
\end{align*}
where  $\kappa':= \alpha^2 p_*  /(\alpha-1)$ with $\alpha:=\frac{1}{p_*}+\frac{1}{p_*}(1-\frac{1}{\rho})\frac{q}{q+1}$, $\rho$ as in Proposition~\ref{prop:sobolev},  and 
\begin{align*}
\cA_5^\om(n)\;:=\;  \norm{1\vee \mu^\omega}_{p,p,Q(n)} \, \norm{1\vee \nu^\omega}_{q,q,Q(n)},
 \quad
\beta_n\;:=\;  \vartheta \sum_{k=0}^{K_n-1}(1-\vartheta)^k,
\end{align*}
with  $\vartheta :=1/(2\alpha p_*) \in(0,1)$ and $K_n:=\big\lfloor\frac{\Delta\ln n-\ln(\sigma-\sigma')}{\ln 2}\big\rfloor$.
\end{proposition}
\begin{proof}
Write $u(\cdot,\cdot)=p^\omega(0,\cdot,0,\cdot)$ and $\sigma_k:=\sigma-(\sigma-\sigma^\prime)2^{-k}$ for $k\in\bbN$.  Then,
\begin{align*}
\norm{u}_{2\alpha p_*,2\alpha p_*,Q_{\sigma_k}}\leq\norm{u}_{1,1,Q_{\sigma_k}}^\vartheta \norm{u}_{\infty,\infty,Q_{\sigma_k}}^{1-\vartheta },
\end{align*}
 by H\"older's inequality. Note that by the definition of $K_n$ we have $\sigma_k-\sigma_{k-1}>n^{-\Delta}$ for all $k\in\{1,\ldots,K_n\}$. By \cite[Theorem 5.5]{ACDS18} (notice that $f=0$ in the present setting which leads to $\gamma=1$ therein), there exist $c=c(d)\in(1,\infty)$, $N_9(\Delta)\in\bbN$ such that for $n\geq N_9(\Delta)$ and $k\in \{1,...,K_n\}$,
\begin{align*}
\norm{u}_{\infty,\infty,Q_{\sigma_{k-1}}} \! \leq  c \, \bigg(\frac{\cA_5^\omega(n)}{(\sigma_k-\sigma_{k-1})^2}\bigg)^{\!\kappa} \norm{u}_{2\alpha p_*,2\alpha p_*, Q_{\sigma_k}} \!  \leq  2^{2\kappa k}  J \, \norm{u}_{1,1,Q_{\sigma_k}}^\vartheta \norm{u}_{\infty,\infty,Q_{\sigma_k}}^{1-\vartheta},
\end{align*}
with $\kappa:=\frac{\alpha}{2(\alpha-1)}$ and $J:=c\left(\frac{\cA_5^\om(n)}{(\sigma-\sigma^\prime)^2}\right)^\kappa\geq1.$ Then by iteration,
\begin{align*}
\norm{u}_{\infty,\infty,Q_{\sigma^\prime}}& \; \leq \; 2^{2\kappa\sum_{k=0}^{K_n-1}(k+1)(1-\vartheta)^k}\left(J\norm{u}_{1,1,Q_\sigma}^\vartheta \right)^{\sum_{k=0}^{K_n-1}(1-\vartheta)^k}\norm{u}_{\infty,\infty,Q_{\sigma_{K_n}}}^{(1-\vartheta)^{K_n}}\\
&\; \leq \;  2^{2\kappa/\vartheta ^2} \, J^{1/\vartheta } \, \norm{u}_{1,1,Q_\sigma}^{\beta_n},
\end{align*}
where we used that $u\leq 1$.
\end{proof}


\begin{assumption}
\label{ass:poly moments dynamic}
Suppose that $\mathbb{E}\big[\omega_0(e)^{2(\kappa'\vee p)}\big]  <  \infty$ and $\mathbb{E}\big[\omega_0(e)^{-2(\kappa'\vee q)}\big] <  \infty$
for any $e\in E_d$ with $p,q\in(1,\infty]$ satisfying \eqref{eq:pq_cond_dyn} and $\kappa'$ as in Proposition~\ref{prop:maximal_dyn}.
\end{assumption}

\begin{proposition}
\label{prop:annealed maximal bound dynamic}
Suppose Assumption \ref{ass:ergodic dynamic} and Assumption \ref{ass:poly moments dynamic} hold. Then for all $K>0$ and $0<T_1\leq T_2$, there exists $c_{15}=c_{15}(d,p,q,K,T_1,T_2)$ such that
$$\E\bigg[\sup_{n\in\N,\,\abs{x}\leq K,\,t\in[T_1,T_2]}n^dp^\omega(0,n^2t,0,\floor{nx})\bigg] \; \leq \; c_{15}.$$
\end{proposition} 
We postpone the proof of the above to the end of this section. First, we deduce the annealed local limit theorem from it.

\begin{proof}[Proof of Theorem~\ref{thm:annealed llt dyn}]
The statement follows from the corresponding quenched result, see Theorem~\ref{thm:quenched_dynamic}-(ii) above, together with Proposition~\ref{prop:annealed maximal bound dynamic} by an application of the dominated convergence theorem. Note that the moment condition in Assumption~\ref{ass:poly moments dynamic} is stronger than the one required in Theorem~\ref{thm:quenched_dynamic}.
\end{proof}


The proof of Proposition~\ref{prop:annealed maximal bound dynamic} begins with a consequence of Proposition~\ref{prop:maximal_dyn}.

\begin{lemma} 
\label{lem:hkL1_dyn}
For all $K>0,\,0<T_1\leq T_2$, there exist $N_{10}=N_{10}(T_2,K)\in\bbN$ and constants $c_{16}=c_{16}(d,p,q,K,T_1,T_2)$, $c_{17}=c_{17}(K,T_2)$ such that for all $n\geq N_{10}$, 
\begin{align*}
\sup_{\abs{x}\leq K,\,t\in\left[T_1,T_2\right]}n^d \, p^\omega(0,n^2t,0,\floor{nx})
\; \leq \; c_{16} \, \cA_5^\omega(c_{17} n)^{\kappa'}
\end{align*}
with $\cA_5^\omega$ and $\kappa'$ as in Proposition~\ref{prop:maximal_dyn}.
\end{lemma}
\begin{proof}

First note that by definition of the heat kernel $p^\omega$,
\begin{align} \label{eq:hk_1norm_dyn}
& \| p^\omega(0,\cdot,0,\cdot)\| _{1,1,Q(n)} \; = \;\frac{1}{\abs{I_1}}\int_{I_1}\frac{1}{\abs{B(n)}}\sum_{y\in B(n)}p^\omega(0,t,0,y)  \, dt \nonumber\\
& 
\mspace{36mu}
= \; c \, n^{-d} \, \frac{1}{\abs{I_1}}\int_{I_1} P_{0,0}^\omega\big[ X_t\in B(n)\big] \, dt  \;\leq \; c \, n^{-d},
\end{align}
for all $n\in\mathbb{N}$. 
Set $x_0=0$, $t_0=T_1$ and let $N=c_{18}\, n$ with $c_{18}$ chosen  such that
\begin{align*}
\big\{(n^2t,\floor{nx}):\, t\in [T_1,T_2],\,\abs{x}\leq K\big\} \;\subseteq \; Q_{1/2}(N)=\big[t_0,t_0+\tfrac{1}{2}N^2\big]\times B(x_0,N/2).
\end{align*}
Then by applying Proposition~\ref{prop:maximal_dyn} with the choice $\Delta=1/2,$ $\sigma=1$ and $\sigma'=1/2$ we get that for all $n\geq \ceil{\frac{N_9}{c_{17}}}\vee 4$,
\begin{align*}
\sup_{\abs{x}\leq K,\,t\in\left[T_1,T_2\right]} n^d \, p^\omega(0,n^2t,0,\floor{nx})
\; \leq \; c \, \cA_5^\omega(c_{17} n)^{\kappa'} \, n^{d(1-\beta_n)}.
\end{align*}
Since $ n^{d(1-\beta_n)} \to 1$ as $n\to \infty$ the claim follows.
\end{proof}

For the proof of Proposition~\ref{prop:annealed maximal bound dynamic} we also require a maximal ergodic theorem for space-time ergodic environments.

\begin{proposition}
\label{prop:dynamic ergodic}
Suppose Assumption \ref{ass:ergodic dynamic} holds. Let $x_0\in\bZ^d,\,t_0\geq 0$ and $p\geq 1$. Then there exists $c_{18}=c_{18}(p)>0$ such that for all $f\in L^p(\Omega)$,
\begin{equation}
\label{eq:dynamic dom ergodic thm}
\E\Bigg[\bigg(\sup_{n\geq1}\frac{1}{n^2}\int_{t_0}^{t_0+n^2}\frac{1}{\abs{B(x_0,n)}}\sum_{x\in B(x_0,n)}f \circ \tau_{t,x} \,dt\bigg)^p\Bigg] \;\leq  \; c_{18} \, \E\left[f^p\right].
\end{equation}
\end{proposition}
\begin{proof}
See the discussion following \cite[Chapter 6, Theorem 4.4, p.224]{krengel}.
\end{proof}

\begin{proof} [Proof of Proposition~\ref{prop:annealed maximal bound dynamic}]
By Lemma~\ref{lem:hkL1_dyn}, it suffices to bound $\bbE\big[\sup_{n\geq 1}\cA_5^\om(c_{17}n)^{\kappa'}\big]$ under the moment condition of Assumption \ref{ass:poly moments dynamic}. This follows by using the maximal ergodic theorem of Proposition~\ref{prop:dynamic ergodic}, similarly to the proof of Proposition~\ref{prop:general speed annealed maximal bound}.
\end{proof}

\section{Applications to the Ginzburg-Landau $\nabla\phi$ Model}
\label{section:interface}

In this section we apply the homogenization results for the dynamic RCM in Theorems~\ref{thm:quenched_dynamic} and \ref{thm:annealed llt dyn} in the context of  a stochastic interface model, the Ginzburg-Landau $\nabla \phi$ model. The survey \cite{Fu05} provides a comprehensive review of this class of models. We write $\Lambda\Subset\bZ^d$ for $\Lambda$ a finite subset of $\bZ^d$. $\Lambda^*$ denotes the set of all undirected edges in $\Lambda$, i.e.\ $\Lambda^*=\lbrace\{x,y\}\in E_d:\,x,y \in\,\overline{\Lambda}\rbrace$, and we write $\mathcal{P}(S)$ for the family of Borel probability measures on some topological space $S$. 

\subsection{Setup and Existence of $\phi$-Gibbs Measures} The Ginzburg-Landau $\nabla\phi$ model describes a hypersurface (interface) embedded in $d+1$-dimensional space, $\bR^{d+1}$, which separates two pure thermodynamical phases. The interface is represented by a field of height variables $\phi=\{\phi(x) \in\bR: x\in\Gamma\}$, which measure the vertical distances between the interface and $\Gamma\subseteq\bZ^d$, a fixed $d$-dimensional reference hyperplane.
%
The Hamiltonian $H$ represents the energy associated with the field of height variables $\phi$. In general, for $\Gamma=\bZ^d$ or $\Gamma\Subset\bZ^d$,
\begin{equation}
\label{eq:hamiltonian}
H(\phi)\equiv H_\Gamma^\psi (\phi)=\sum_{\{x,y\} \in\Gamma^*}V(\phi(x)-\phi(y)).
\end{equation}
Note that boundary conditions $\psi=\{\psi(x):x\in\,\partial^+\Gamma\}$ are required to define the sum in the case $\Gamma\Subset\bZ^d$, i.e.\ we set $\phi(x)=\psi(x)$ for $x\in\partial^+\Gamma.$ The sum in \eqref{eq:hamiltonian} is merely formal  when $\Gamma=\bZ^d$. 
%
The dynamics of the $\nabla\phi$ model are governed by the following infinite system of SDEs for $\phi_t=\lbrace\phi_t(x):\,x\in\Gamma\rbrace\in\,\bR^\Gamma,$
$$d\phi_t(x)=-\frac{\partial H}{\partial \phi(x)}(\phi_t)\,dt+\sqrt{2}\,dw_t(x),\quad x\in\Gamma,\; t>0,$$
where $w_t=\{w_t(x) : x\in \mathbb{Z}^d\}$ is a collection of independent one-dimensional standard Brownian motions. Due to the form of the Hamiltonian, only nearest neighbour interactions are involved. Equivalent to the above in the case $\Gamma=\bZ^d$ is
\begin{align} \label{eq:phi_dyn}
\phi_t(x)=\phi_0(x)-\int_0^t \sum_{y:|x-y|=1} V'(\phi_t(x)-\phi_t(y)) \, dt + \sqrt{2}\, w_t(x), \qquad x\in \bZ^d.
\end{align}
Similarly, if $\Gamma\Subset\bZ^d$, we define the finite volume process by
\begin{align*} 
\phi^{\Gamma,\psi}_t(x)=\phi_0^{\Gamma,\psi}(x)-\int_0^t \sum_{y\in\overline{\Gamma}:|x-y|=1} V'\big(\phi_t^{\Gamma,\psi}(x)-\phi_t^{\Gamma,\psi}(y)\big) \, dt + \sqrt{2}\, w_t(x), \quad x\in \Gamma,
\end{align*}
subject to the boundary conditions $\phi_t^{\Gamma,\psi}(y)=\psi(y),\;y\,\in\,\partial^+\Gamma$. The evolution of $\phi_t$ is designed such that it is stationary and reversible under the equilibrium $\phi$-Gibbs measure $\mu_\Gamma^\psi$ or $\mu$ (see \eqref{eq:finite gibbs} below).
We denote $\bP_{\mu}$ the law of the process $\phi_t$ started under the distribution $\mu$ (and $\E_\mu$ the corresponding expectation). By a slight abuse of notation we will also write  $\E_{\mu}$, $\var_\mu$ and $\cov_\mu$  for the expectation, variance and  covariance  under $\mu$.

Most of the mathematical literature on the $\nabla \phi$ model treats the case of a suitably smooth, even and strictly convex interaction potential $V$ such that $V''$ is bounded above. However, we will relax these conditions; throughout the rest of this section we work with $V$ as in the following assumption.

\begin{assumption}
\label{ass:potential}
The potential $V\,\in\,C^2(\R)$ is even and there exists $c_->0$ such that
\begin{equation}
\label{eq:lower convexity}
c_{-}\leq V''(x),\quad \text{for all }x\in\R.
\end{equation}
\end{assumption}
Note that under Assumption \ref{ass:potential}, the coefficients of the SDE \eqref{eq:phi_dyn} are not necessarily globally Lipschitz continuous. However, it is still possible to construct an almost surely continuous solution $\phi_t$, see Proposition \ref{prop:sde solution}.
The assumption that the potential has second derivative bounded away from zero is helpful for the existence of an equilibrium $\phi$-Gibbs measure.
%
%
For $\Gamma\Subset \bZ^d$, the finite volume $\phi$-Gibbs measure for the field of heights $\phi\in\bR^d$ is defined as
\begin{equation}
\label{eq:finite gibbs}
\mu(d\phi)\equiv \mu_\Gamma^\psi(d\phi)=\frac{1}{Z_\Gamma^\psi}\exp\left(-H_\Gamma^\psi(\phi)\right)d\phi_\Gamma,
\end{equation}
with boundary condition $\psi\in\bR^{\partial^+\Gamma}$, where $d\phi_\Gamma$ is the Lebesgue measure on $\bR^\Gamma$ and $Z_\Gamma^\psi$ is a normalisation constant.
Then  \eqref{eq:lower convexity} implies $Z_\Gamma^\psi<\infty$ for every $\Gamma\Subset\bZ^d$ and hence $\mu_\Gamma^\psi\in\mathcal{P}(\bR^\Gamma)$ is a probability measure. In the infinite volume case $\Gamma=\bZ^d$, \eqref{eq:finite gibbs} has no rigorous meaning but one can still define Gibbs measures as follows.

\begin{definition}
A probability measure $\mu\in\mathcal{P}\big(\bR^{\bZ^d}\big)$ is a $\phi$-Gibbs measure if its conditional probability on $\mathcal{F}_{\Gamma^c}=\sigma\lbrace\phi(x): x\notin\Gamma\rbrace$ satisfies the DLR (Dobrushin-Lanford-Ruelle) equation
\begin{equation}
\label{eq:DLR}
\mu(\cdot\vert\mathcal{F}_{\Gamma^c})(\psi)=\mu_\Gamma^\psi(\cdot),\quad\text{for }\mu\text{-a.e. }\psi,
\end{equation}
for all $\Gamma\Subset\bZ^d$.
\end{definition} 


In order to study the properties of solutions to the system of SDEs \eqref{eq:phi_dyn}, it is necessary to restrict to a suitable class of initial configurations. Let $\cS:=\{(\phi(x))_{x\in\bbZ^d}:\abs{\phi(x)}\leq a+\abs{x}^n,\text{ for some }a\in\bbR,\,n\in\bbN\}$ denote the configurations of heights with at most polynomial growth.

\begin{proposition}
\label{prop:sde solution}
Given any initial configuration $\phi_0\in\cS$, there exists a unique solution to the system of SDEs \eqref{eq:phi_dyn} such that for any $x\in\bbZ^d$ the process $\phi_t(x)$ is almost surely continuous and for all $t>0$ the configuration $\phi_t\in\cS$ almost surely. Any Gibbs measure concentrated on $\cS$ is stationary and reversible with respect to the process $\phi_t$.
\end{proposition}
\begin{proof}
The proof follows by similar arguments as for the Ising model case of \cite[Theorem 4.2.13]{Roy07}. The key observations are that equation (4.2.5) there holds for our Hamiltonian and the relation (4.2.12b) holds for our interaction potential $V$, where $p$ is defined as $p(x)=c_-\mathbbm{1}_{\abs{x}=1}$ for $x\in\bbZ^d$.
\end{proof}


Brascamp-Lieb inequalities state that for $\Gamma\Subset\bZ^d$, covariances under the aforementioned $\phi$-Gibbs measure $\mu_\Gamma^\psi$ are bounded by those under $\mu_\Gamma^{\psi,\,G}$, the Gaussian finite volume $\phi$-Gibbs measure determined by the quadratic potential $V^*(x)=\frac{c_-}{2}x^2$. 

\begin{proposition}[Brascamp-Lieb inequality for exponential moments]
\label{prop:brascamp lieb}
Let $\Gamma\Subset\bZ^d$.  For every $\nu\in\bR^\Gamma$, 
\begin{align}
\label{eq:brascamp}
\E_{\mu_\Gamma^\psi}\bigg[\exp\Big(\lvert\langle\nu,\phi-\E_{\mu_\Gamma^{\psi}}[\phi]\rangle_{\ell^2(\Gamma)}\rvert\Big)\bigg]&\leq 2\exp\Big(\frac{1}{2}\var_{\mu_\Gamma^{\psi,G}}\big(\langle\nu,\phi\rangle_{\ell^2(\Gamma)}\big)\Big).
\end{align}
\end{proposition}

\begin{proof}
See \cite[Theorem 4.9]{Fu05}. Note that the condition $V''(x)\leq c_+,\;\forall \, x\in\bR$, for some $c_+>0$, is not needed for the proof.
\end{proof}

This inequality is pivotal in proving the following existence result, which constitutes the first part of Theorem \ref{thm:cov_lim}. We shall also employ the massive Hamiltonian
\begin{equation}
\label{eq:massive hamiltonian}
H_m(\phi)\equiv H_{\Gamma,m}^\psi (\phi)\; := \; H^\psi_\Gamma \,+ \, \frac{m^2}{2}\sum_{x\in\Gamma}\phi(x)^2, \qquad m>0.
\end{equation}

\begin{remark}
Note that Proposition~\ref{prop:brascamp lieb} also holds for the massive Hamiltonian $H_{\Gamma, m}^\psi$ and in that case the Gaussian potential can be taken to be $V^*(x)=\frac{c_-+2d m^2}{2}x^2$.
\end{remark}

\begin{theorem}[Existence of $\phi$-Gibbs measures]
\label{thm:phi gibbs existence}
If $d\geq3$ then for all $h\in\bR$ there exists a stationary, shift-invariant, ergodic $\phi$-Gibbs measure $\mu\in\cP\big(\bR^{\bZ^d}\big)$ of mean $h$, i.e.\ $\E_{\mu}[\phi(x)]=h$ for all $x\in\bZ^d$.
\end{theorem}
\begin{proof}
Let $m>0$ and first take a sequential limit as $n\to\infty$ of finite volume $\phi$-Gibbs measures $\mu^0_{m,\, \Gamma_n}$ with periodic boundary conditions, corresponding to the massive Hamiltonian $H_{\Gamma_n,\,m}^0$ on the torus $\Gamma_n:=\left(\bZ/n\bZ\right)^d$. Since $V$ is even, $\bbE_{\mu^0_{m,\, \Gamma_n}}\left[\phi(x)\right]=0$. When $d\geq 3$, the variance of the Gaussian system corresponding to the potential $V^*(x)=\frac{c_-}{2}x^2$ is uniformly bounded in $n$, so by the Brascamp-Lieb inequality,
$$\sup_{x\in\bZ^d,\,n\in\bbN}\,\bbE_{\mu^0_{m,\, \Gamma_n}}\left[\exp(\lambda\abs{\phi(x)})\right]<\infty,\quad \forall\,\lambda>0.$$
In order to show tightness of $(\mu^0_{m,\, \Gamma_n})_{n\in\bbN}$ we introduce a weighted $l^2(\bbZ^d)$ norm 
$$\norm{\phi}_{r}^2:=\sum_{x\in\bbZ^d}\phi(x)^2\me^{-2r\,\abs{x}}\qquad \text{for }r>0.$$
Then sets of the form $K_M=\{\phi\in\bbR^{\bbZ^d} :  \norm{\phi}_r\leq M\}$ are compact,  cf.\  e.g.\ \cite[proof of Proposition~3.3]{Fu05}. For any $n\in \bbN$,
\begin{align*}
\bbP_{\mu^0_{m,\, \Gamma_n}}\big(K_M^c\big)&\leq \bbE_{\mu^0_{m,\, \Gamma_n}}\Big[\norm{\phi}_r^2/M^2\Big]=\frac{1}{M^2}\sum_{x\in\bbZ^d}\bbE_{\mu^0_{m,\, \Gamma_n}}\Big[\phi(0)^2\Big]\me^{-2r\,\abs{x}},
\end{align*}
where we have used shift-invariance of $\mu^0_{m,\, \Gamma_n}$. Then the above Brascamp-Lieb inequality implies
$$\sup_{n\in \bbN}\, \bbP_{\mu^0_{m,\, \Gamma_n}}\big(K_M^c\big)\leq c\, \sum_{x\in\bbZ^d}\frac{\me^{-2r\,\abs{x}}}{M^2}\leq \frac{c}{M^2}.$$
This tends to zero as $M\to \infty$. Therefore $(\mu^0_{m,\, \Gamma_n})_{n\in\bbN}$ is tight and along some proper subsequence there exists a limit $\mu^0_m:=\lim_{k\to\infty}\mu^0_{m,\, \Gamma_{n_k}}$, a shift-invariant Gibbs measure on $\bbZ^d$ of mean $0$. Now for all $m>0$ and $x\in\bbZ^d$, $\frac{\partial^2 H_m(\phi)}{\partial \phi(x)^2}\geq c_-$ so by the Brascamp-Lieb inequality again (taking an infinite volume limit)
$$\sup_{x\in\bZ^d,\,0 < m \leq 1}\,\bbE_{\mu^0_m}\left[\exp(\lambda\abs{\phi(x)})\right]<\infty,\quad \forall\,\lambda>0.$$
With this bound we can then show that the limit $\mu^0=\lim_{m\downarrow0}\mu_m^0$ exists, analogously to the above argument. The distribution of $\phi+h$ where $\phi$ is $\mu^0$ distributed is a shift-invariant $\phi$-Gibbs measure on $\bbZ^d$ under which $\phi(x)$ has mean $h$ for all $x\in\bZ^d$. Having shown that the convex set of shift-invariant $\phi$-Gibbs measures of mean $h$ is non-empty,  there exists an extremal element of this set which is ergodic, see \cite[Theorem 14.15]{Ge11}. Finally, by Proposition \ref{prop:sde solution} this Gibbs measure is reversible and hence stationary for the process $\phi_t$.
\end{proof}

\begin{remark}
The $\phi$-Gibbs measures exist when $d\geq 3$ but not for $d=1,\,2$. An infinite volume (thermodynamic) limit for $\mu_\Gamma^0$ as $\Gamma\uparrow\bZ^d$ exists only when $d\geq 3$.
\end{remark}

\subsection{Proofs of Theorems~\ref{thm:cov_lim} and \ref{thm:equilibrium scaling} via Helffer-Sj\"ostrand representation }
Our first aim is to investigate the decay of the space-time correlation functions under the equilibrium Gibbs measures. The idea, originally from Helffer and Sj\"ostrand \cite{HS}, is to describe the correlation functions in terms of a certain random walk in a dynamic random environment (cf.\ also \cite{DD, GOS}).
Let $(X_t)_{t\geq 0}$ be the random walk on $\Z^d$ with jump rates given by the random dynamic conductances
\begin{align} \label{eq:def_env}
\omega_t(e):=V''(\nabla_e \phi_t)=V''(\phi_t(y)-\phi_t(x)), \qquad e=\{x,y\}\in E_d.
\end{align}
Note that the conductances are positive by Assumption \ref{ass:potential} and, since $V$ is even, the jump rates are symmetric, i.e.\ $\omega_t(\{x,y\})=\omega_t(\{y,x\})$. Further, let $p^{\om}(s,t,x,y)$, $x,y\in \Z^d$, $s\leq t$, denote the transition densities of the random walk $X$. Then the Helffer-Sj\"ostrand representation (see \cite[Theorem 4.2]{Fu05} or \cite[Equation (6.10)]{DD}) states that if $F,\,G\in C_b^1(\cS)$ are differentiable functions with bounded derivatives depending only on finitely many coordinates then for all $t>0$,
\begin{equation}
\label{eq:general HS}
\cov_{\mu}(F(\phi_0), G(\phi_t))=\int_0^\infty \!\!\sum_{x,\,y\in\bbZ^d}\mathbb{E}_{\mu} \bigg[ \frac{\partial F(\phi_0)}{\partial \phi(x)} \,\frac{\partial G (\phi_{t+s})}{\partial \phi(y)}\, p^{\om}(0,t+s,x,y) \bigg] \,ds,
\end{equation}
where $\mu$ is a stationary, ergodic, shift-invariant $\phi$-Gibbs measure. Note that for $d\geq 3$ the integral in \eqref{eq:general HS} is finite due to the following on-diagonal heat kernel estimate.

\begin{lemma}
\label{lemma:on-diagonal estimate}
There exists deterministic $c_{19}=c_{19}(d,c_-)<\infty$ such that 
\begin{equation}
p^\om(0,t,x,y)\, \leq \, c_{19} \,t^{-d/2}, \qquad \forall\, t\geq 1,\,  x,y\in \bbZ^d.
\end{equation}
\end{lemma}
\begin{proof}
Note that by Assumption~\ref{ass:potential},
$\om_t (e)\,\geq\, c_-$ for all $t\geq 0$ and $e\in E_d$, which implies the Nash inequality, i.e.\ for any $f\,:\,\bbZ^d\,\to\,\bbR$,
\begin{align*}
\norm{f}_{l^2(\bbZ^d)}\,\leq\, c\,\norm{\om_t^{1/2} \nabla f}_{l^2(E_d)}^{\frac{d}{d+2}} \, \norm{f}_{l^1(\bbZ^d)}^{\frac{2}{d+2}},
\end{align*}
from which the statement follows by standard arguments, see \cite{CKS87} and \cite{MO16}.
\end{proof}
A consequence of the above is the following variance estimate, an example of algebraic decay to equilibrium, in contrast to the  exponential decay to equilibrium which would follow from a spectral gap estimate or Poincar\'e inequality. For this model, these inequalities hold on finite boxes but fail on the whole lattice.

\begin{corollary}
\label{cor:variance decay}
Suppose $d\geq 3$ and let $\mu\,\in\,\cP\big(\bR^{\bZ^d}\big)$ be any ergodic, shift-invariant, stationary $\phi$-Gibbs measure. 
\begin{enumerate}
\item[(i)]  There exists $c_{20}=c_{20}(d,c_-)>0$ such that for all $F,\,G\in C_b^1(\cS)$ and $t>0$,
\begin{align*}
\abs{\cov_{\mu}\big(F(\phi_0), G(\phi_t)\big)} \leq  \frac{c_{20}}{(t\vee 1)^{\frac{d}{2}-1}}\sum_{x,\,y\in\bbZ^d}\mathbb{E}_{\mu} \bigg[ \Big(\frac{\partial F}{\partial \phi(x)}(\phi_0)\Big)^{\!2} \bigg]^{\frac 1 2}\mathbb{E}_{\mu} \bigg[ \Big( \frac{\partial G}{\partial \phi(y)}(\phi_0) \Big)^{\!2} \bigg]^{\frac 1 2}.
\end{align*}
\item[(ii)] $\bbE_{\mu}\left[\phi(0)^2\right]\,<\,\infty\,$. 
\end{enumerate}

\end{corollary}
\begin{proof}
(i) follows by the Cauchy-Schwarz inequality applied to \eqref{eq:general HS} together with Lemma \ref{lemma:on-diagonal estimate} and stationarity of $\mu$. Further, by taking $t \downarrow 0$ we deduce from (i) and dominated convergence that $\sup_M \bbE_\mu\left[(\phi_0(0)\wedge M) ^2\right]<\infty$, which gives (ii).   	
\end{proof}


\begin{proof}[Proof of Theorem~\ref{thm:cov_lim}]
In Theorem~\ref{thm:phi gibbs existence} above, the existence of a stationary, shift-invariant, ergodic $\phi$-Gibbs measure $\mu$ has been shown. Further, the environment $\omega$ defined in \eqref{eq:def_env} satisfies Assumption~\ref{ass:ergodic dynamic} by the ergodicity of $\mu$. 
Note that $\om_t(e)\geq c_-$  for any $e\in E_d$ and $t>0$
by Assumption \ref{ass:potential}, so we may set $q=\infty$ in Assumption~\ref{ass:poly moments dynamic}, which then reduces to \eqref{eq:moment_gradphi}. The Helffer-Sj\"ostrand relation  \eqref{eq:general HS} gives
\begin{align*}
\cov_{\mu}(\phi_0(0), \phi_t(x))=\int_0^\infty \mathbb{E}_{\mu} \left[ p^{\om}(0,t+s,0,x) \right] \,ds.
\end{align*}
Now, applying Theorem \ref{thm:annealed llt dyn},
\begin{align}
\nonumber
n^{d-2} \cov_{\mu}(\phi_0(0), \phi_{n^2t}\left(\lfloor nx \rfloor)\right)= \, & \int_0^\infty  \mathbb{E}_{\mu}\left[ n^d\,p^{\om}\big(0,n^2(t+s),0,\lfloor nx\rfloor\big) \right]\,ds \\
\label{eq:covariance convergence}
\xrightarrow [n\to \infty]{} & \int_0^\infty k_{t+s}(x)\, ds,
\end{align}
which is the claim. Note that Theorem~\ref{thm:annealed llt dyn} gives uniform convergence of the integrand on any compact interval $[0,T]$ and Lemma~\ref{lemma:on-diagonal estimate} tells us that the integrand is dominated by $c_{20} s^{-\frac{d}{2}}$, integrable on $[T,\infty)$ since $d\geq 3$. Therefore, by dominated convergence  we are justified in interchanging the limit and the integral.
\end{proof}

Having applied the RCM local limit theorem to prove the above space-time covariance scaling limit, we now present an application of the invariance principle in Theorem~\ref{thm:quenched_dynamic}-(i). We use this to characterise the scaling limit of the equilibrium fluctuations as a Gaussian free field. Recall that for $f\in C_0^\infty (\bbR^d)$ and $n\in\bbZ_+$,
\begin{align*} 
f_n(x):=n^{-(1+d/2)}f(x/n)
 \quad \text{and} \quad   \phi(f_n):=n^{-(1+d/2)} \int_{\bbR^d}f(x) \, \phi(\floor{nx})\,dx.
\end{align*}

\begin{proof}[Proof of Theorem~\ref{thm:equilibrium scaling}]
Let $n\in\bbZ_+$ and set $G_n(\lambda):=\bbE_\mu[\exp(\la\phi(f_n))]$, $\la\in \bbR$. By the Brascamp-Lieb inequality Proposition~\ref{prop:brascamp lieb}, $\{G_n(\la)\}_{n\geq 1}$ is uniformly bounded for $\la$ in a compact interval. So differentiating gives
\begin{align*}
\frac{dG_n(\la) }{d\la}=\bbE_\mu\big[\phi(f_n)\exp(\la \phi(f_n))\big].
\end{align*}
Recall that $(X_t)_{t\geq 0}$ denotes the random walk on $\bbZ^d$ under the conductances given by \eqref{eq:def_env}. For simplicity, we write $P_x$ for the law of the walk started from $x\in \bbZ^d$ at time $s=0$ and $E_x$ for the corresponding expectation. Then by the Helffer-Sj\"ostrand relation \eqref{eq:general HS}, the above equals
\begin{align*}
&\la\, n^{-(d+2)}\int_0^\infty \sum_{x\in\bbZ^d}\bbE_\mu\Big[f(x/n)\, E_x\big[f(X_t/n \big)\big] \exp(\la\phi_t(f_n))\Big]\,dt\\
=\,&\la\, n^{-d}\int_0^\infty\sum_{x\in\bbZ^d}f(x/n)\,\bbE_\mu\Big[ E_x\big[f(X_{n^2t}/n \big)\big] \exp(\la\phi_{n^2t}(f_n))\Big]\,dt
\end{align*}
where we have employed a simple change of variable.  The key step of the proof is to show that
\begin{align} \label{eq:BSconv}
\lim_{n\to\infty} n^{-d} \int_0^\infty \sum_{x\in\bbZ^d} f(x/n) \, E_x\big[ f(X_{n^{2}t}/n) \big]\, dt= \int_{\bbR^d} f(x)(Q^{-1}f)(x)\,dx\quad\text{in }L^2(\mu).
\end{align}
Once this is proven, we can deduce that
\begin{align*}
\frac{dG_n(\la)}{d\la}=\la\,G_n(\la)\int_{\bbR^d} f(x)(Q^{-1}f)(x)\,dx + \lambda\, o(1),
\end{align*}
as $n\to \infty$. This relies on the fact that by stationarity of the Gibbs measure, $\bbE_\mu[\exp(\la\phi_{n^2t}(f_n))]=\bbE_\mu[\exp(\la\phi(f_n))]$ for any $\lambda$ and $t\geq 0$, and this term is uniformly bounded in $n$. Therefore, denoting $G(\la)=\lim_{n\to\infty}G_n(\la)$, we have by letting $n\to\infty$,
$$\frac{dG(\la)}{d\la}=\la\,G(\la)\int_{\bbR^d} f(x)(Q^{-1}f)(x)\,dx$$
which together with $G(0)=1$ gives the claim. 

Thus, it remains to show \eqref{eq:BSconv}. This follows from the same arguments as for the corresponding result in \cite[Proposition~4.4]{BS11}, cf.\ \cite[Proposition~14]{NW18} also. We omit the details here but since the right hand side of \eqref{eq:BSconv} can also be written as a time integral, the main idea is to decompose it into an integral over $[0,T]$ and $[T,\infty)$, for some suitably large $T$. We then obtain the desired convergence from the integral over $[0,T]$ by applying the annealed functional central limit theorem for the random walk $X$, which is an immediate consequence of the QFCLT in Theorem~\ref{thm:quenched_dynamic}-(i) with $q=\infty$. The remaining integral over $[T,\infty)$ can be neglected due to heat kernel decay of order $t^{-d/2}$ in Lemma~\ref{lemma:on-diagonal estimate}. Note that in dimension $d\geq 3$, as explained in \cite[Remark~4.6]{BS11}, such a heat kernel estimate may serve as a replacement for the decay estimate on the spatial derivative of the heat kernel in \cite[Lemma~3.6]{BS11}.
\end{proof}

\subsection{Moments of  $\phi$-Gibbs measures}
Finally, we shall derive Proposition~\ref{prop:moments} giving polynomial moment bounds on the heights $\phi$ under any ergodic, shift-invariant, stationary $\phi$-Gibbs measure. Hence we can verify the conditions in Theorem~\ref{thm:cov_lim} and Theorem~\ref{thm:equilibrium scaling} for any polynomial potential satisfying Assumption~\ref{ass:potential}. The proof will require the following comparison estimate for $\phi_t$ and $\phi_t^{L_n}$ where $L_n:=[-n,n]^d\cap\bbZ^d$ for $n\in\bbN$.

\begin{lemma}
\label{lemma:sde estimate}
Let $\mu$ be a shift-invariant Gibbs measure. There exists a positive, symmetric, summable sequence $\alpha=\big(\alpha(x)\big)_{x\in\bbZ^d}$ such that the following holds. There exist constants $c_{21},\,c_{22},\,c_{23} > 0$ such that for any $n \in  \bbN$, $t>0$ and any bounded Lipschitz function $f$ on $l^2(\bbZ^d,\alpha)\,$,
\begin{equation}
\label{eq:sde estimate 1}
\bbE_\mu\Big[\big(f\big(\phi_t\big)-f\big(\phi_t^{L_n,0}\big)\big)^2\Big]\, \leq \,c_{21}\, \norm{f}_{\lip,\alpha}^2\,\me^{c_{22} t-c_{23} n}\, \Big(1+\bbE_\mu\big[\phi_0(0)^2\big]\Big),
\end{equation}
where $\norm{f}_{\lip,\alpha}:=\sup_{\phi \neq \phi'}\abs{f(\phi)-f(\phi')}\norm{\phi-\phi'}_{l^2(\bbZ^d,\alpha)}^{-1}$.
\end{lemma}
\begin{proof}
By the same arguments as the Ising model case in \cite{Roy07} (see (4.2.14), (5.1.5) and the proof of Corollary 5.1.4 therein) we have for any $c_{22}>1+8dc_-$,
\begin{align*}
\alpha(0)\,\bbE_\mu\Big[\big(f\big(\phi_t\big)-f\big(\phi_t^{L_n,0}\big)\big)^2\Big]\,
 \leq \,\me^{c_{22} t}\, \norm{f}_{\lip,\alpha}^2 \bbE_\mu\Big[ \sum_{x \not\in L_n} \alpha*\alpha(x)\,\phi_0(x)^2+c\abs{\alpha}\alpha(x)\Big],
\end{align*}
where in our setting $\alpha:=\sum_{k=0}^\infty (\sigma')^{-k}(*p)^k$ is constructed from the sequence $p(x)=c_-\mathbbm{1}_{\abs{x}=1}$ for any $\sigma'\,> \, 2dc_-$. Then by the shift-invariance of $\mu$ and exponential decay of $\alpha$ and $\alpha*\alpha$ on $\bbZ^d$, we get \eqref{eq:sde estimate 1} (cf.\ also \cite[Section~1.1]{Zi08}).
\end{proof}

\begin{proof}[Proof of Proposition \ref{prop:moments}]
Since $\mu$ is stationary and shift-invariant it suffices to show that $\bbE_{\mu}\big[\abs{\phi_0(0)}^p\big]<\infty$ for all $p>0$. By Jensen's inequality it is enough to consider $p >  2$. For any  $M>1$ let $f_M(\phi):=\big(\abs{\phi(0)}\wedge M\big)^{p/2}$ which is Lipschitz continuous on $l^2(\bbZ^d,\alpha)$ with  $\norm{f_M}_{\lip,\alpha}^2\,\leq \, c\,M^{p-2}$. For arbitrary $t\, >\, 0$ and $n\, \in\, \bbN$,
\begin{align}
\bbE_{\mu}\left[f_M(\phi_t)^2\right]
\label{eq:moment triangle}
&\leq \, \bbE_{\mu}\Big[\Big(f_M\big(\phi_t\big)-f_M\big(\phi_t^{L_n,0}\big)\Big)^{\!2} \Big]+\bbE_{\mu}\left[f_M\big(\phi_t^{L_n,0}\big)^2\right].
\end{align}
To control the first term on the right hand side of \eqref{eq:moment triangle}, we  fix $\eps >  0$. As argued in \cite[Theorem 5.1.3]{Roy07}, for arbitrary $\lambda> 0 $ we introduce an increasing sequence of boxes $L_{n(t)}$ such that $c_{24}t\,\leq\, n(t)\, \leq \, c_{24}(t+1)$ where $c_{24}>0$ is chosen such that $c_{22}t-c_{23}n(t)\,<\,-\lambda t$, with $c_{22},\,c_{23}$ as in Lemma \ref{lemma:sde estimate}. Therefore, by \eqref{eq:sde estimate 1} and Corollary~\ref{cor:variance decay}-(ii), there exists $T_{\eps,\,M} > 0$ such that
\begin{align}\label{eq:p moment sde term}
\bbE_{\mu}\Big[\Big(f_M\big(\phi_t\big)-f_M\big(\phi_t^{L_{n(t)},0}\,\big)\Big)^{\!2}\Big] \, \leq \, c\,M^{p-2}\Big(1+\bbE_{\mu}\left[\phi_0(0)^2\right]\Big)\me^{-\lambda t} \leq\, \eps
\end{align}
for all $t> T_{\eps,\,M}$. For the latter term in \eqref{eq:moment triangle}, the constant zero boundary condition allows us, via the DLR equation \eqref{eq:DLR}, to reduce the expectation to that over a finite Gibbs measure as defined in \eqref{eq:finite gibbs},
\begin{equation}
\bbE_{\mu}\Big[f_M \big(\phi_t^{L_n,0}\big)^2\Big]\, = \, \bbE_{\mu_{L_n}^0}\Big[f_M \big(\phi_t^{L_n,0}\big)^2\Big].
\end{equation}
Now, the finite volume process $\phi^{L_n,0}$ is stationary with respect to $\mu_{L_n}^0$ so by the Brascamp-Lieb inequality, as argued in Theorem~\ref{thm:phi gibbs existence},
\begin{equation}
\label{eq:p moment fin term}
\sup_{n\in\bbN,\,M>0,\,t>0}\bbE_{\mu_{L_n}^0}\Big[f_M \big(\phi_t^{L_n,0}\big)^2\Big]=\sup_{n\in\bbN,\,M>0}\bbE_{\mu_{L_n}^0}\Big[f_M\big(\phi_0^{L_n,0}\big)^2\Big]<\infty.
\end{equation}
Substituting \eqref{eq:p moment fin term} and \eqref{eq:p moment sde term} into \eqref{eq:moment triangle} gives
\begin{equation}
\label{eq:moment M bound}
\bbE_{\mu}\Big[\big(\abs{\phi_t(0)}\wedge M\big)^p\Big] \, < \, \eps + c,
\end{equation}
for all $t\, > \, T_{\eps, M}$, with the constant $c$ independent of $M$. However, $\phi_t$ is stationary with respect to $\mu$ so \eqref{eq:moment M bound} in fact holds for all $t\, \geq \, 0$. We conclude by the monotone convergence theorem, letting $M\uparrow \infty$, that
$\bbE_{\mu}\big[\abs{\phi_0(0)}^p\big]  < \infty$.
\end{proof}

\subsubsection*{Acknowledgements}
We thank the referee for the careful reading and constructive suggestions. We also thank Alberto Chiarini and Martin Slowik for discussions.

\bibliographystyle{abbrv}
\bibliography{refs}

\end{document}